\documentclass{article}
\usepackage{amsmath}
\usepackage{tocbibind}
\usepackage{stmaryrd}   
\usepackage{amsmath}
\usepackage{verbatim}
\usepackage{mathrsfs}
\usepackage{mathtools}
\usepackage{amsthm} 
\usepackage{amssymb} 
\usepackage{mdframed}
\usepackage[utf8]{inputenc}
\usepackage[english]{babel}
\usepackage{booktabs}
\usepackage[colorlinks={true},linkcolor={blue},citecolor={blue}, urlcolor={blue} ]{hyperref}
\usepackage{cleveref}
\usepackage[margin=2.5 cm]{geometry}
\usepackage[textsize=footnotesize]{todonotes}
\usepackage[toc,page]{appendix}
\usepackage{tikz}
\usetikzlibrary{calc,intersections,through,hobby}
\tikzset{point/.style={circle,inner sep=0pt,minimum size=4pt,fill=black}}
\tikzset{pointG/.style={circle,inner sep=0pt,minimum size=4pt,fill=gray}}

\usepackage{booktabs}
\usepackage{capt-of}

\theoremstyle{definition}
\newtheorem{definition}{Definition}[section]
\theoremstyle{plain}

\theoremstyle{plain}
\newtheorem{theorem}[definition]{Theorem}
\theoremstyle{plain}
\newtheorem{corollary}[definition]{Corollary}
\theoremstyle{plain}
\newtheorem{proposition}[definition]{Proposition}
\theoremstyle{remark}
\newtheorem{question}[definition]{Question}
\theoremstyle{remark}

\theoremstyle{remark}
\newtheorem{remark}[definition]{Remark}
\theoremstyle{remark}

\theoremstyle{remark}

\newcommand{\QQ}{\mathbb Q}
\newcommand{\ZZ}{\mathbb{Z}}

\newcommand{\CC}{\mathbb{C}}
\newcommand{\PP}{\mathbb{P}}
\newcommand{\FF}{\mathbb{F}}
\newcommand{\HH}{\mathcal{H}}

\DeclareMathOperator{\spec}{Spec}
\DeclareMathOperator{\Divi}{Div}
\DeclareMathOperator{\supp}{supp}
\DeclareMathOperator{\CH}{CH}

\title{Intersection matrices for the minimal regular model of ${X}_0(N)$ and applications to the Arakelov canonical sheaf}

\author{Paolo Dolce \and Pietro Mercuri}
\date{}

\newcommand{\Addresses}{{
  \bigskip
  \footnotesize
}
  
  P.~Dolce, \textsc{Ben Gurion University of the Negev, Be'er Sheva, Israel}\par\nopagebreak
  \textit{E-mail address}: \texttt{dolce@bgu.ac{.}il}

  \medskip
  
P.~Mercuri, \textsc{Dipartimento SBAI , ``Sapienza'' Università di Roma, Rome, Italy}\par\nopagebreak
  \textit{E-mail address}: \texttt{mercuri.ptr@gmail.com}

}

\begin{document}

\maketitle
\begin{abstract}
Let $N>1$ be an integer coprime to $6$ such that $N\notin\{5,7,13\}$ and let $g=g(N)$ be the  genus of  the modular curve $X_0(N)$. We compute the intersection matrices  relative to special fibres of the minimal regular model of $X_0(N)$. Moreover  we prove that the self-intersection of the Arakelov canonical sheaf of $X_0(N)$ is asymptotic to $3g\log N$, for $N\to+\infty$.
\end{abstract}
\makeatletter

\@starttoc{toc}
\makeatother

\section{Introduction}
\subsection{Presentation of the results}
The self-intersection of the Arakelov canonical sheaf of an arithmetic surface, here denoted by $\langle\overline\omega,\overline\omega\rangle$, is a crucial arithmetic invariant for several reasons (assuming that the  genus of the generic fibre is at least $2$):
\begin{itemize}
\item It appears in the arithmetic Noether's formula proved in \cite{MB89}. So it is closely related  to the Faltings height of the Jacobian of the generic fibre of the arithmetic surface.

\item  The inequality $\langle\overline\omega,\overline\omega\rangle>0$ is equivalent to the arithmetic Bogomolov conjecture for arithmetic surfaces proved by Ullmo and Zhang in \cite{Ull98} and \cite{Zha93}. 

\item Suitable upper bounds  for $\langle\overline\omega,\overline\omega\rangle$ imply an effective version of the Mordell conjecture (see for instance \cite{MB90}).
\end{itemize}

The study of the quantity $\langle\overline\omega,\overline\omega\rangle$ is in general tremendously difficult since it requires a deep knowledge of the specific model of the curve under investigation together with some analytic information involving the curve seen as a point inside the compactified moduli space. Nevertheless, some partial results are known for Fermat curves, hyperelliptic curves, Belyi curves and modular curves (see for instance  \cite{dJ04},\cite{CK09}, \cite{ECdJMB11}, \cite{K13}, \cite{Jav14}).

 The exact value  of $\langle\overline\omega,\overline\omega\rangle$ for minimal regular models of modular curves is currently not known, but it makes sense to try to figure out what is the asymptotic behaviour of $\langle\overline\omega,\overline\omega\rangle$ when the level grows. This problem has been partially investigated for all the ``classical'' modular curves $X_0(N)$, $X_1(N)$, $X(N)$. In particular if we denote with $g=g(N)$ the genus of curve, it turned out that\footnote{The notation $\sim$ between positive real valued sequences reads as ``is asymptotic to''. In symbols: $a_N\sim b_N$,  for $N\to\infty$ if $\lim_{N\to\infty}\frac{a_N}{b_N}=1$.} 
\begin{equation}\label{eq:fund_as}
\langle\overline\omega,\overline\omega\rangle\sim 3g\log N,\quad \text{ for } N\to+\infty,
\end{equation}
in the following cases:
\begin{itemize}
\item For $X_0(N)$ when $N$ is square-free with $(N,6)=1$ (see \cite{AU97} and \cite{MU98}) and moreover when $N\in\{p^2,p^3,p^4\}$ for $p>5$  prime (see \cite{BBC20} and \cite{BMC22}).
\item For $X_1(N)$ when $N$ is at the same time square-free, odd and divisible by at least two coprime integers bigger than $4$. (see \cite{May14}). 
\item For $X(N)$ when $N$ is at the same time square-free, composite and odd (see \cite{GvP22}).
\end{itemize}

In this paper we restrict our attention to the case of  $X_0(N)$ and its minimal regular model $\mathcal X$, which has singular and not reduced fibres only above the primes dividing the level $N$. In order to calculate the ``finite contribution'' of  $\langle\overline\omega,\overline\omega\rangle$, for every $p\mid N$ one needs the matrix containing all the intersection numbers between the irreducible components of the fibre $\mathcal X_p$ over $p$. We compute such matrices for every $N>1$  coprime to $6$ such that $N\notin\{5,7,13\}$. The assumption $(N,6)=1$ on the level is substantial, since at the moment we don't have well behaved regular  models  of $X_0(N)$ when either $2\mid N$ or $3\mid N$. On the other hand, the restriction $N\notin\{5,7,13\}$ is just a matter of convenience, since these cases must be treated with slightly different methods. In the literature  the intersection matrices of the minimal regular models of $X_0(N)$  had been explicitly computed only for  $N$ at the same time coprime to $6$ and squarefree (see \cite{DR73} and \cite[Appendix]{Maz77}) and for $N\in\{p^2,p^3,p^4\}$ and $p$ a prime (see \cite{Edi90} and \cite{BMC22}). 
Moreover, we use the intersection matrices of $\mathcal X$ to study the asymptotics of $\langle\overline\omega,\overline\omega\rangle$ for the modular curves $X_0(N)$, when $N$ is coprime to $6$ extending (as expected) the aforementioned results. In particular we prove the following theorem:
\begin{theorem}\label{thm:main_res}
Let $N>1$ be an integer coprime to $6$ and let $\overline\omega$ be the Arakelov canonical sheaf of the  minimal regular model of $X_0(N)$, then
\[
\langle\overline\omega,\overline\omega\rangle\sim 3g\log N,\quad \text{ for } N\to+\infty.
\]
\end{theorem}

\subsection{Overview of the paper}
We follow the approach of \cite{BBC20} and \cite{BMC22} for the case $N\in\{p^2,p^3,p^4\}$, but with several refinements and improvements. Let $N>1$ be an integer coprime to $6$; first of all we make use of the Edixhoven's model for $X_0(N)$ (see \cite{Edi90}): it is a regular model $\mathcal X'\to\spec \ZZ$ with the property that for every prime $\ell\nmid N$ the fibre $\mathcal X'_\ell$ over $\ell$ is smooth and irreducible. Above the primes dividing $N$ we find the so called \emph{special fibres} that have many irreducible components, that can be completely described. The Edixhoven's model is in general not minimal and in some cases it is necessary to perform several blow downs in order to obtain the minimal regular model $\mathcal X$. Each intersection matrix of the special fibres of $\mathcal X$ is the symmetric square matrix whose entries are the intersection numbers of the components of the special fibre. They appear as  the coefficient matrices of some linear systems whose $\mathbb Q$-solutions are the multiplicities of two special rational vertical divisors $V_0$ and $V_\infty$ supported on the special fibres of $\mathcal X$. By using the Faltings-Hriljac's version of the Hodge's index theorem and the Manin-Drinfeld's theorem it is possible to write $\langle\overline\omega,\overline\omega\rangle$ in the following way:
\begin{equation*}
\langle\overline{\omega}, \overline{\omega}\rangle=\underbrace{-4g(g-1)\langle H_0, H_\infty\rangle}_{(a)}\underbrace{+\frac{g\langle V_0, V_\infty\rangle}{g-1}-\frac{\langle V_0, V_0\rangle+\langle V_\infty, V_\infty\rangle}{2g-2}}_{(b)}\underbrace{+\frac{h_0+h_\infty}{2}}_{(c)},
\end{equation*}
where:
\begin{itemize}
\item[$(a)$] is a multiple of the Arakelov intersection between $H_0$ and $H_\infty$, which are the prolongation on $\mathcal X$ of the two cusps $0$ and $\infty$ of $X_0(N)$;
\item[$(b)$] contains only the intersection between $V_0$ and $V_\infty$ and their respective self-intersections;
\item[$(c)$] is a ``height term'' related to the Néron-Tate's height of the restriction of some  specific  horizontal divisors  on the generic fiber.
\end{itemize}
The pieces $(a)$ and $(c)$ contain the data relative to the ``self-intersection at infinity'' of $\overline \omega$. Their asymptotics have already been estimated respectively in \cite{MvP22} and \cite[Section 6]{MU98}. In this paper we are able to exactly calculate the term $(b)$. The main difficulty consisted in finding an explicit expression for the divisors $V_0$ and $V_\infty$ for a general $N$ since the dimension of the matrix of the above mentioned linear systems depends on $N$. We observed that the exact solutions of these linear systems, written in a suitable way, have a ``simple'' shape that we proved is the correct formula for every level $N$ using the software Mathematica. Our method for finding $V_0$ and $V_\infty$ is completely general and depends on the Zariski's lemma for vector spaces (see \cite[Lemma 2.2.1]{Mor13}). Whereas in \cite{BBC20} and \cite{BMC22} the authors use some \emph{ad hoc} techniques. Also the final computation of the summand (b) is performed by using the software Mathematica.

The structure of the paper is the following: In \Cref{sec:prel} we fix the notation and we briefly recall the needed basic concepts. In \Cref{sec:Edi} we give a detailed description of the Edixhoven's model and we compute the intersection matrices. In \Cref{sec:non_min-edi} we discuss the process of obtaining the minimal regular model, when necessary, and, in \Cref{sec:min}, we give the intersection matrices for these minimal models. Finally, \Cref{sec:main_res} is devoted to the proof of \Cref{thm:main_res}. Moreover  Appendix \ref{sec:app} contains the drawings of the special fibers of the Edixhoven's model. In the github repository \cite{codes} we provide the codes of Mathematica we used for the computations.

\paragraph{Acknowledgements}
The authors are deeply grateful to Professor \emph{P. Parent} for helping them with some computations and for his suggestions. A special thanks goes also to Professor \emph{A. Chambert-Loir} for his kind comments. This research project started on 2021 when both the authors where postdoc researchers at the University of Udine.

\section{Preliminaries}\label{sec:prel}

\subsection{Arakelov theory}\label{sec:arakelov}
In this section we briefly recall the basic notions of Arakelov geometry on arithmetic surfaces, for more details the reader can consult the original Arakelov paper \cite{Ara74} together with Faltings seminal paper \cite{Fal84}, or  the more recent books \cite{Lan88} and \cite{Mor14}. For our purposes, it is enough to deal with the theory for curves defined over $\mathbb Q$, but everything can be easily written down when the ground field is any number field, in fact it is enough to add more places at infinity.

Let $X$ be a smooth, geometrically integral,  projective curve over $\mathbb Q$ of genus $g>0$, and let $f\colon\mathcal X\to \spec \mathbb Z$ be any regular model of $X$. On $\mathcal X$ it is not possible to define a meaningful intersection pairing between divisors that descends to the Picard group. Nevertheless, there is  a ``partial'' $\mathbb Z$-valued intersection pairing which is well defined when at least one of the divisors is supported over a prime  $p$ (i.e., contained in one fiber of $f$). For every $D\in\Divi(\mathcal X)$ and every $D'\in\Divi(\mathcal X)$ such that $\supp (D)\subseteq \mathcal X_p$ we denote this pairing by $D\cdot D'$; for further details and properties the reader can check \cite[Chapter 9]{Liu06}. In order to intersect horizontal divisors we need to take into account the base change curve $X_{\mathbb C}$ as a fiber lying over the Archimedean place of $\mathbb Z$, i.e., as a 
\emph{fiber at infinity}. At this point one needs to fix a K\"ahler form $\Omega$ on $X_{\mathbb C}$, and in general the intersection numbers depend on such a choice. Arakelov proposed a canonical choice for this K\"ahler form which here we denote by $\Omega^{\text{can}}$ (see for instance \cite[page 112]{Mor14}). One can extend the notion of line bundle on $\mathcal X$  and introduce the concept of \emph{admissible hermitian line bundle} which is a couple $\overline{\mathcal L}=(\mathcal L, h)$ where $\mathcal L$ is a line bundle on $\mathcal X$ and $h$ is a $C^{\infty}$-hermitian metric on the base change $\mathcal L_{\mathbb C}$ (on $X_{\mathbb C}$) satisfying:
\[
c_1(\overline{\mathcal L})=a\Omega^{\text{can}},\quad \text{for } a\in\mathbb R,
\]
where $c_1(\overline{\mathcal L})$ is the first Chern form of $\overline{\mathcal L}$. By using the Deligne pairing and the notion of Arakelov degree for hermitian line bundles on $\spec\mathbb Z$ one can define a $\mathbb R$-valued intersection pairing between admissible hermitian line bundles which is denoted by $\langle \overline{\mathcal L}, \overline{\mathcal L}'\rangle$.

There is a map that relates admissible hermitian line bundle to \emph{Arakelov divisors} which are the pairs $\overline D= (D,\alpha)$ where $D\in\Divi(\mathcal X)$ and $\alpha\in\mathbb R$. Here it is important to mention that the real constant $\alpha$ appearing in $\overline D$ can be recovered as the integral of a Green's function over $X_{\mathbb C}$. The Arakelov divisors form a group under the obvious notion of addition and one can form the Arakelov-Chow group $\overline{\CH}^1(X)$ by performing the quotient by an adequate notion of principal Arakelov divisors. In \cite{Ara74} is defined a $\mathbb R$-valued bilinear and symmetric intersection pairing among Arakelov divisors denoted as $\langle \overline D, \overline D'\rangle$. We point out that an ordinary divisor  $D\in \Divi(\mathcal X)$ can be identified with the Arakelov divisor $(D,0)$ and, therefore  the notation  $\langle D,  D'\rangle$ makes sense. In particular, if $D$ is any divisor and $D'$ is a divisor supported over the prime $p$ and we have the following fundamental relationship:
\[
\langle D,  D'\rangle=(D\cdot D')\log p. 
\]
A crucial result of the theory  (see for instance \cite[Sections 4.3 and 4.4]{Mor14}) shows that the aforementioned map between admissible line bundles and Arakelov divisors descends to a group isomorphism between the isometry classes of admissible line bundles and  $\overline{\CH}^1(X)$ and moreover preserves  the intersection pairings.

The relative canonical sheaf $\omega$ of $f\colon\mathcal X\to\spec \mathbb Z$ can be endowed with a canonical hermitian metric so that we obtain the \emph{Arakelov canonical sheaf $\overline{\omega}$} that satisfies an arithmetic version of the adjunction formula (see for instance \cite[Section 4.5]{Mor14}). Let $\mathcal K$ be a canonical divisor of $\mathcal X$ i.e., every divisor whose corresponding line bundle is isomorphic to $\omega$, then one can easily prove  that as an Arakelov divisor $\mathcal K$ corresponds to $\overline{\omega}$ through the above mentioned isomorphism. Therefore we conclude that $\langle\mathcal K,\mathcal K\rangle=\langle\overline \omega,\overline \omega\rangle$. Let $\Delta$ be diagonal of $X_{\mathbb C}\times X_{\mathbb C}$, then there exists a unique (symmetric) $C^\infty$-function $\mathcal G\colon (X_{\mathbb C}\times X_{\mathbb C})- \Delta\to\mathbb R$ such that:
\begin{itemize}
\item[$(i)$] around any point $P\in X_{\mathbb C}$ we can write $\mathcal G(P,\cdot)=-\log\vert z\rvert^2+u$, where $z$ is a  chart centered in $P$ and $u$ is a $C^{\infty}$-function;
\item[$(ii)$] $\frac{1}{2\pi i}\partial\bar\partial \mathcal G(P,\cdot)=\Omega^{\text{can}}$;
\item[$(iii)$] $\int_{X_{\mathbb C}}\mathcal G(P,\cdot)\Omega^{\text{can}}=0$.
\end{itemize}
This function is called the \emph{(canonical) Green's function} on $X_{\mathbb C}$ and it is crucial to compute the contribution at infinity of the Arakelov intersection pairing. For instance if $\overline P$ and $\overline Q$ are disjoint Arakelov divisors that are the closures on $\mathcal X$ respectively of $P,Q\in X(\mathbb Q)$, then one can show that 
\[
2\langle \overline P, \overline Q  \rangle= -\mathcal G(P,Q).
\]
\subsection{Modular curves}

In this subsection we explain our notation and recall some basic facts about modular curves. Our main reference for this subsection is \cite{DS05}. Let $\HH:=\{ \tau \in \CC: \mathrm{Im}(\tau)>0 \}$ be the complex upper half-plane, let $\PP^1(\QQ)$ be the set of cusps and denote by $\HH^*:=\HH\cup\PP^1(\QQ)$ the extended complex upper half-plane. There is an action of $\mathrm{SL}_2(\ZZ)$ on $\HH^*$ given, for $\left(\begin{smallmatrix}a&b\\ c&d\end{smallmatrix}\right)\in\mathrm{SL}_2(\ZZ)$ and $\tau\in\HH^*$, by
\[
\begin{pmatrix}a&b\\ c&d\end{pmatrix}\tau:=\frac{a\tau+b}{c\tau+d}.
\]
Let $N$ be a positive integer, the subgroup 
\[
\Gamma(N):=\{\left(\begin{smallmatrix}a&b\\ c&d\end{smallmatrix}\right)\in\mathrm{SL}_2(\ZZ):\left(\begin{smallmatrix}a&b\\ c&d\end{smallmatrix}\right)\equiv \left(\begin{smallmatrix}1&0\\ 0&1\end{smallmatrix}\right) \mod N\}
\]
is called \emph{principal congruence subgroup of level $N$}. Each subgroup $\Gamma$ of $\mathrm{SL}_2(\ZZ)$ containing $\Gamma(N)$ is called \emph{congruence subgroup of level $N$}. We define the modular curve  over $\CC$ associated to a congruence subgroup $\Gamma$ in the following way:
\[
X_\Gamma:=\Gamma\backslash \HH^*.
\]
We also define, for every positive integer $N$, the congruence subgroup
\[
\Gamma_0(N):=\{\left(\begin{smallmatrix}a&b\\ c&d\end{smallmatrix}\right)\in\mathrm{SL}_2(\ZZ):\left(\begin{smallmatrix}a&b\\ c&d\end{smallmatrix}\right)\equiv\left(\begin{smallmatrix}*&*\\ 0&*\end{smallmatrix}\right) \mod N\},
\]
(where any $*$ in the matrices means  ``no restriction") and we define the associated modular curve:
\[
X_0(N):=\Gamma_0(N)\backslash \HH^*.
\]
The curve $X_0(N)$ is defined over $\CC$, but it can be defined over $\QQ$ as well (see \cite[Chapter~7]{DS05}).

\begin{remark}\label{rem:genus}
The genus $g$ of $X_0(N)$ is given, for $N\ge 3$, by
\[
g=1+\frac{d(N)}{12}-\frac{\varepsilon_2(N)}{4}-\frac{\varepsilon_3(N)}{3}-\frac{\varepsilon_\infty(N)}{2},
\]
with $d$, $\varepsilon_2$, $\varepsilon_3$ and $\varepsilon_\infty$ multiplicative functions given, for a prime $p$, by
\begin{align*}
d(p^n)&=p^{n-1}(p+1),\\
\varepsilon_2(p^n)&=\begin{cases}
0, & \text{if }p=2 \text{ and } n\ge 2, \\
1+\left(\frac{-1}{p}\right), & \text{otherwise}, \\
\end{cases} \\
\varepsilon_3(p^n)&=\begin{cases}
0, & \text{if }p=3 \text{ and } n\ge 2, \\
1+\left(\frac{-3}{p}\right), & \text{otherwise},
\end{cases} \\
\varepsilon_\infty(p^n)&=\begin{cases}
p^{\frac{n}{2}-1}(p+1), & \text{if }n\text{ is even}, \\
2p^{\frac{n-1}{2}}, & \text{if }n\text{ is odd},
\end{cases}
\end{align*}
where $\left(\frac{\cdot}{p}\right)$ is the Kronecker symbol (see for example \cite[Section~3.9, pag.~107]{DS05}). In our case, i.e., $N>1$ coprime with $6$, we have that
\begin{equation}\label{eq:ogenus}
g=\frac{d(N)}{12}+o(N), \quad \text{for }N\to+\infty.
\end{equation}
\end{remark}
The modular curve $X_0(N)=\Gamma_0(N)\backslash \HH^*$ can be seen as the set of $\CC$-points of a coarse moduli space where each point $\Gamma_0(N)\tau$, for $\tau\in\HH$, represents a pair $(E,C)$, up to a suitable equivalence relation, where $E$ is an elliptic curve over $\CC$ and $C$ is cyclic subgroup of order $N$ of $E$. Two pairs $(E,C)$ and $(E',C')$ are equivalent if $E$ and $E'$ are isomorphic and this isomorphism sends $C$ to $C'$ (see \cite[Section 1.5]{DS05}). The reduction modulo a prime $\ell\nmid N$ is compatible with the moduli space structure by Igusa's theorem (see \cite[Section 8.6]{DS05}), hence in a pair $(E,C)$ the curve $E$ is just defined over $\FF_\ell$. The reduction modulo a prime $p\mid N$ is more complicated (see \cite{KM85}) but some points can still be seen as pairs $(E,C)$ where the curve $E$ is defined over $\FF_p$. We call a point on $X_0(N)$ reduced modulo a prime, dividing or not the level $N$, \emph{supersingular} if it can be seen as a pair $(E,C)$ such that $E$ is a supersingular elliptic curve over the corresponding finite field, i.e., an elliptic curve whose endomorphism ring is isomorphic to an order of a quaternion algebra of dimension 4 over the finite field. The $j$-invariant of a supersingular elliptic curve is called a supersingular $j$-invariant.

\section{Regular models of $X_0(N)$}\label{sec:minmod}

 In \cite{Edi90}, Edixhoven found a regular model of the curve $X_0(N)$, for $N$ coprime to $6$. We denote Edixhoven's model by $\mathcal X'$. In most of the cases this regular model is minimal. But in some cases some blow downs are necessary to get the minimal regular model (see \Cref{sec:non_min-edi} below for more details about this). In both cases we denote by $\mathcal X$ the minimal regular model of $X_0(N)$. In this paper we need to study the geometry of the fibres of $\mathcal X'$ and $\mathcal X$ which can be not reduced, so from now on with the word ``genus''(of a curve) we tacitly mean ``arithmetic genus''.

\subsection{The Edixhoven's model}\label{sec:Edi}
\begin{center}
\framebox{In this is subsection we fix $N>1$ and $(N,6)=1$.}
\end{center}

\noindent Here we describe in a very detailed way the Edixhoven's model $\mathcal X'$; the related drawings can be found in Appendix \ref{sec:app}. For a prime $\ell\nmid N$, the fiber $\mathcal X'_\ell$ of $\mathcal X'$ at $\ell$ is smooth (see \cite[Theorem~8.6.1 and the discussion below]{DS05}) and it is isomorphic to $X_0(N)/\FF_\ell$. The fiber $\mathcal X'_p$ at $p$ a prime dividing $N$ is more complicate. Let $N=p^nM$, with $p\nmid M$. The fiber $\mathcal X'_p$ has $n+1$ \emph{Igusa components} denoted by $C'_a$ and indexed by an integer $a\in\{0,1,\ldots,n\}$. Each $C'_a$ is a curve isomorphic to $X_0(M)/\FF_p$ with multiplicity $\phi(p^{\min(a,n-a)})$, where $\phi$ is the Euler's totient function. 
Let
\begin{equation}\label{eq:xi}
\xi(m):=\frac{1-\left(\frac{m}{p}\right)}{2},
\end{equation}
where $\left(\frac{\cdot}{p}\right)$ is the Kronecker symbol, in particular we only use the following two special cases:
\begin{align*}
\xi(-1)=\frac{2-\varepsilon_2(p^n)}{2}=\begin{cases}
0, & \text{if }p\equiv 1 \pmod 4, \\
1, & \text{if }p\equiv 3 \pmod 4,
\end{cases} \quad \text{and} \quad \xi(-3)=\frac{2-\varepsilon_3(p^n)}{2}=\begin{cases}
0, & \text{if }p\equiv 1 \pmod 3, \\
1, & \text{if }p\equiv 2 \pmod 3.
\end{cases}
\end{align*}
There are
\begin{equation}\label{eq:k}
k:=\frac{p-1}{12}d(M)-\frac{1}{2}\xi(-1)\varepsilon_2(M)-\frac{1}{3}\xi(-3)\varepsilon_3(M)
\end{equation}
points where each Igusa component intersects all the others (the notation is the same as \Cref{rem:genus}). Among these $k$ points, $\left\lfloor\frac{p}{12}\right\rfloor$ correspond to the supersingular $j$-invariants modulo $p$ for $j\notin\{0,1728\}$. The remaining (if any) comes from the ramification points over the possibly supersingular $j$-invariant $j\in\{0,1728\}$. Moreover, there are some \emph{extra components} whose number and properties depend on $p\bmod{12}$ and are related to the elliptic points of $X_0(M)$.

\begin{description}

\item[Case $p\equiv 1 \pmod{12}$.] In this case for each $C'_a$ with $a\notin\{0,n\}$, there are $\varepsilon_2(M)$ extra components denoted by $E'_{a,i}$, with $i=1,\ldots,\varepsilon_2(M),$ such that each $E'_{a,i}$ has multiplicity $\frac{1}{2}\phi(p^{\min(a,n-a)})$ and intersects only $C'_a$ at the $i$-th elliptic point of $X_0(M)$ corresponding to the $j$-invariant $j=1728$. Moreover, for each $C'_a$ with $a\notin\{0,n\}$, there are $\varepsilon_3(M)$ extra components denoted by $F'_{a,i}$, with $i=1,\ldots,\varepsilon_3(M),$ such that $F'_{a,i}$ has multiplicity $\frac{1}{3}\phi(p^{\min(a,n-a)})$ and intersects only $C'_a$ at the $i$-th elliptic point of $X_0(M)$ corresponding to the $j$-invariant $j=0$. See \Cref{fig:fig1} in Appendix \ref{sec:app}.

\item[Case $p\equiv 5 \pmod{12}$.] In this case for each $C'_a$ with $a\notin\{0,n\}$, there are $\varepsilon_2(M)$ extra components denoted by $E'_{a,i}$, with $i=1,\ldots,\varepsilon_2(M),$ such that each $E'_{a,i}$ has multiplicity $\frac{1}{2}\phi(p^{\min(a,n-a)})$ and intersects only $C'_a$ at the $i$-th elliptic point of $X_0(M)$ corresponding to the $j$-invariant $j=1728$. Moreover, if $n$ is even there are $\varepsilon_3(M)$ extra components denoted by $F'_{\infty,i}$, with $i=1,\ldots,\varepsilon_3(M),$ each with multiplicity $\frac{1}{3}\varepsilon_\infty(p^n)$ that intersects in a single point all the $C'_a$ for $0\le a<\frac{n}{2}$, intersects in a separate point $C'_{\frac{n}{2}}$ and intersects in a separate single point all the $C'_a$ for $\frac{n}{2}<a\le n$; these intersection points on the Igusa components correspond to the $i$-th elliptic point of $X_0(M)$ associated to the $j$-invariant $j=0$. If $n$ is odd there are $2\varepsilon_3(M)$ extra components denoted by $F'_{0,i}$ and $F'_{n,i}$, with $i=1,\ldots,\varepsilon_3(M),$ each with multiplicity $\frac{1}{2}\varepsilon_\infty(p^n)$ and such that $F'_{0,i}$ intersects $F'_{n,i}$ (same $i$) in one point, $F'_{0,i}$ intersects in a single point every $C'_a$ for $0\le a<\frac{n}{2}$ and $F'_{n,i}$ intersects in a single point every $C'_a$ for $\frac{n}{2}<a\le n$; these intersection points on the Igusa components correspond to the $i$-th elliptic point of $X_0(M)$ associated to the $j$-invariant $j=0$. See \Cref{fig:fig2,fig:fig3} in Appendix \ref{sec:app}.

\item[Case $p\equiv 7 \pmod{12}$.] In this case there are $\varepsilon_2(M)$ extra components denoted by $E'_{\infty,i}$, with $i=1,\ldots,\varepsilon_2(M),$ each with multiplicity $\frac{1}{2}\varepsilon_\infty(p^n)$ and that intersects in a single point all the $C'_a$ for $0\le a<\frac{n}{2}$, intersects in a separate single point all the $C'_a$ for $\frac{n}{2}<a\le n$ and intersects (only when $n$ is even) in a separate point $C'_{\frac{n}{2}}$; these intersection points on the Igusa components correspond to the $i$-th elliptic point of $X_0(M)$ associated to the $j$-invariant $j=1728$. Moreover, for each $C'_a$ with $a\notin\{0,n\}$, there are $\varepsilon_3(M)$ extra components denoted by $F'_{a,i}$, with $i=1,\ldots,\varepsilon_3(M),$ such that $F'_{a,i}$ has multiplicity $\frac{1}{3}\phi(p^{\min(a,n-a)})$ and intersects only $C'_a$ at the $i$-th elliptic point of $X_0(M)$ corresponding to the $j$-invariant $j=0$. See \Cref{fig:fig4} in Appendix \ref{sec:app}.

\item[Case $p\equiv 11 \pmod{12}$.] In this case there are $\varepsilon_2(M)$ extra components denoted by $E'_{\infty,i}$, with $i=1,\ldots,\varepsilon_2(M),$ each with multiplicity $\frac{1}{2}\varepsilon_\infty(p^n)$ and that intersects in a single point all the $C'_a$ for $0\le a<\frac{n}{2}$, intersects in a separate single point all the $C'_a$ for $\frac{n}{2}<a\le n$ and intersects  (only when $n$ is even) in a separate point $C'_{\frac{n}{2}}$; these intersection points on the Igusa components correspond to the $i$-th elliptic point of $X_0(M)$ associated to the $j$-invariant $j=1728$. Moreover, if $n$ is even there are $\varepsilon_3(M)$ extra components denoted by $F'_{\infty,i}$, with $i=1,\ldots,\varepsilon_3(M),$ each with multiplicity $\frac{1}{3}\varepsilon_\infty(p^n)$ that intersects in a single point all the $C'_a$ for $0\le a<\frac{n}{2}$, intersects in a separate point $C'_{\frac{n}{2}}$ and intersects in a separate single point all the $C'_a$ for $\frac{n}{2}<a\le n$; these intersection points on the Igusa components correspond to the $i$-th elliptic point of $X_0(M)$ associated to the $j$-invariant $j=0$. If $n$ is odd there are $2\varepsilon_3(M)$ extra components denoted by $F'_{0,i}$ and $F'_{n,i}$, with $i=1,\ldots,\varepsilon_3(M),$ each with multiplicity $\frac{1}{2}\varepsilon_\infty(p^n)$ and such that $F'_{0,i}$ intersects $F'_{n,i}$ (same $i$) in one point, $F'_{0,i}$ intersects in a single point every $C'_a$ for $0\le a<\frac{n}{2}$ and $F'_{n,i}$ intersects in a single point every $C'_a$ for $\frac{n}{2}<a\le n$; these intersection points on the Igusa components correspond to the $i$-th elliptic point of $X_0(M)$ associated to the $j$-invariant $j=0$. See \Cref{fig:fig5,fig:fig6} in Appendix \ref{sec:app}.
\end{description}

Hence, we have that:
\begin{itemize}
\item if $p\equiv 1 \pmod{12}$,
\[
\mathcal X'_p=C'_0+C'_n+\sum_{a=1}^{n-1}\phi(p^{\min(a,n-a)})\left(C'_a+\frac{1}{2}\sum_{i=1}^{\varepsilon_2(M)}E'_{a,i}+\frac{1}{3}\sum_{i=1}^{\varepsilon_3(M)}F'_{a,i}\right);
\]
\item if $p\equiv 5 \pmod{12}$ and $n$ is even,
\[
\mathcal X'_p=C'_0+C'_n+\sum_{a=1}^{n-1}\phi(p^{\min(a,n-a)})\left(C'_a+\frac{1}{2}\sum_{i=1}^{\varepsilon_2(M)}E'_{a,i}\right)+\frac{1}{3}\varepsilon_\infty(p^n)\sum_{i=1}^{\varepsilon_3(M)}F'_{\infty,i};
\]
\item if $p\equiv 5 \pmod{12}$ and $n$ is odd,
\[
\mathcal X'_p=C'_0+C'_n+\sum_{a=1}^{n-1}\phi(p^{\min(a,n-a)})\left(C'_a+\frac{1}{2}\sum_{i=1}^{\varepsilon_2(M)}E'_{a,i}\right)+\frac{1}{2}\varepsilon_\infty(p^n)\sum_{i=1}^{\varepsilon_3(M)}\left(F'_{0,i}+F'_{n,i}\right);
\]
\item if $p\equiv 7 \pmod{12}$,
\[
\mathcal X'_p=C'_0+C'_n+\sum_{a=1}^{n-1}\phi(p^{\min(a,n-a)})\left(C'_a+\frac{1}{3}\sum_{i=1}^{\varepsilon_3(M)}F'_{a,i}\right)+\frac{1}{2}\varepsilon_\infty(p^n)\sum_{i=1}^{\varepsilon_2(M)}E'_{\infty,i};
\]
\item if $p\equiv 11 \pmod{12}$ and $n$ is even,
\[
\mathcal X'_p=C'_0+C'_n+\sum_{a=1}^{n-1}\phi(p^{\min(a,n-a)})C'_a+\varepsilon_\infty(p^n)\left(\frac{1}{2}\sum_{i=1}^{\varepsilon_2(M)}E'_{\infty,i}+\frac{1}{3}\sum_{i=1}^{\varepsilon_3(M)}F'_{\infty,i}\right);
\]
\item if $p\equiv 11 \pmod{12}$ and $n$ is odd,
\[
\mathcal X'_p=C'_0+C'_n+\sum_{a=1}^{n-1}\phi(p^{\min(a,n-a)})C'_a+\frac{1}{2}\varepsilon_\infty(p^n)\left(\sum_{i=1}^{\varepsilon_2(M)}E'_{\infty,i}+\sum_{i=1}^{\varepsilon_3(M)}\left(F'_{0,i}+F'_{n,i}\right)\right).
\]
\end{itemize}
By \cite{Edi90}, we know the local equation of each component. Let $P$ be a point not corresponding to an elliptic point of $X_0(M)$, then the local equation of $C'_a$ near $P$ is
\[
\begin{cases}
x-y^{p^n}=0,&\text{if }a=0, \\
x^{p^n}-y=0,&\text{if }a=n, \\
(x^{p^{a-1}}-y^{p^{n-a-1}})^{p-1}=0,&\text{if }a\in\{1,\ldots,n-1\}.
\end{cases}
\]
Let $P$ be a point corresponding to an elliptic point over $j=1728$. If $a<\frac{n}{2}$, then the local equation near $P$ of $E'_{a,i}$ is $t=0$, for every $i\in\{1,\ldots,\varepsilon_2(M)\}$, and the local equation near $P$ of $C'_a$ is
\[
u-t^{\frac{p^{n-2a}-1}{2}}=0,
\]
where $t=y^2$ and $u=\frac{x}{y}$. If $a>\frac{n}{2}$, then the local equation near $P$ of $E'_{a,i}$ is $s=0$, for every $i\in\{1,\ldots,\varepsilon_2(M)\}$, and the local equation near $P$ of $C'_a$ is
\[
s^{\frac{p^{2a-n}-1}{2}}-v=0,
\]
where $s=x^2$ and $v=\frac{y}{x}$. Let $P$ be a point corresponding to an elliptic point over $j=0$. If $a<\frac{n}{2}$, then the local equation near $P$ of $F'_{a,i}$ is $t=0$, for every $i\in\{1,\ldots,\varepsilon_3(M)\}$, and the local equation near $P$ of $C'_a$ is
\[
u-t^{\frac{p^{n-2a}-1}{3}}=0,
\]
where $t=y^3$ and $u=\frac{x}{y}$. If $a>\frac{n}{2}$, then the local equation near $P$ of $F'_{a,i}$ is $s=0$, for every $i\in\{1,\ldots,\varepsilon_3(M)\}$, and the local equation near $P$ of $C'_a$ is
\[
s^{\frac{p^{2a-n}-1}{3}}-v=0,
\]
where $s=x^3$ and $v=\frac{y}{x}$. 

Recalling that if $A$ is a local algebra over a field $\kappa$ with maximal ideal $\mathfrak m$ and $B$ is a finitely generated $A$-module, we have that
\[
\mathrm{length}_A(B)\dim_\kappa (A/\mathfrak m)=\dim_\kappa(B).
\]
See for example \cite[Exercise~1.6 (c), Chapter 7]{Liu06}. Hence, in order to find the local intersection numbers we compute
\[
\dim_{\FF_p}A/(f_1,f_2),
\]
where $P$ is the point of intersection, $A=\mathcal O_{\mathcal X',P}$ and $f_1=0$ and $f_2=0$ are the local equations at $P$ of the two prime components considered. Then we sum these numbers for each intersection point of the two prime components. For instance, if $a<a'<n/2$ and $P$ is a supersingular point that is not an elliptic point, we have that
\begin{align*}
A=\mathcal O_{\mathcal X',P}=\FF_p[x,y]_{(x,y)},
\end{align*}
$C'_a$ has multiplicity $(p-1)p^{a-1}$ and local equation at $P$
\[
x-y^{p^{n-2a}}=0,
\]
$C'_{a'}$ has multiplicity $(p-1)p^{a'-1}$ and local equation at $P$
\[
x-y^{p^{n-2a'}}=0.
\]
Hence
\begin{align*}
A/(f_1,f_2)=\FF_p[x,y]_{(x,y)}/(x-y^{p^{n-2a}},x-y^{p^{n-2a'}}).
\end{align*}
Since in $A$ we have
\begin{align*}
\left(x-y^{p^{n-2a}},x-y^{p^{n-2a'}}\right)&=\left(y^{p^{n-2a'}}-y^{p^{n-2a}},x-y^{p^{n-2a'}}\right)=\left(y^{p^{n-2a'}}(1-y^{p^{2a'-2a}}),x-y^{p^{n-2a'}}\right)= \\
&=\left(y^{p^{n-2a'}},x-y^{p^{n-2a'}}\right)=\left(y^{p^{n-2a'}},x\right),
\end{align*}
then
\begin{align*}
\dim_{\FF_p}\FF_p[x,y]_{(x,y)}/(y^{p^{n-2a'}},x)=p^{n-2a'}.
\end{align*}
Finally, since we have $k$ supersingular points that are not elliptic points (see \Cref{eq:k}), we obtain $C'_a\cdot C'_{a'}=kp^{n-2a'}$ when $p\equiv 1\pmod{12}$. In a similar way we get all the intersection numbers: Let $\xi$ be defined as in \Cref{eq:xi}, let $k$ be defined as in \Cref{eq:k} and let
\begin{align*}
\mu(a,a')&:=\min(|n-2a|,|n-2a'|),
\end{align*}
then the intersection numbers among the components of $\mathcal X_p'$ are the following:
\begin{itemize}
\item for $a,a'\in\{0,\ldots,n\}$ with $a\ne a'$, we have if $(n-2a)(n-2a')> 0$ that
\[
C'_a\cdot C'_{a'}=kp^{\mu(a,a')}+\frac{1}{2}\xi(-1)\varepsilon_2(M)(p^{\mu(a,a')}-1)+\frac{1}{3}\xi(-3)\varepsilon_3(M)\left(p^{\mu(a,a')}-\frac{3-(-1)^n}{2}\right),
\]
and if $(n-2a)(n-2a')\le 0$ that
\[
C'_a\cdot C'_{a'}=k;
\]
\item for $a\in\{0,\ldots,n\}$, $a'\in\{1,\ldots,n-1,\infty\}$, $i\in\{1,\ldots,\varepsilon_2(M)\}$, we have
\[
C'_a\cdot E'_{a',i}=\begin{cases}
1, & \text{if }a'=a \text{ or }a'=\infty, \\
0, & \text{otherwise};
\end{cases}
\]
\item for $a\in\{0,\ldots,n\}$, $a'\in\{0,\ldots,n,\infty\}$, $i\in\{1,\ldots,\varepsilon_3(M)\}$, we have
\[
C'_a\cdot F'_{a',i}=\begin{cases}
1, &\begin{cases}
\text{if }a'=a \\
\text{or }a'=\infty \\
\text{or }a'=0 \text{ and }n-2a>0 \\
\text{or }a'=n \text{ and }n-2a<0,
\end{cases} \\
0, & \text{otherwise};
\end{cases}
\]
\item for $a,a'\in\{1,\ldots,n-1,\infty\}$, $i,i'\in\{1,\ldots,\varepsilon_2(M)\}$, with $(a,i)\ne (a',i')$, we have
\[
E'_{a,i}\cdot E'_{a',i'}=0;
\]
\item for $a\in\{1,\ldots,n-1,\infty\}$, $a'\in\{0,\ldots,n,\infty\}$, $i\in\{1,\ldots,\varepsilon_2(M)\}$, $i'\in\{1,\ldots,\varepsilon_3(M)\}$,
\[
E'_{a,i}\cdot F'_{a',i'}=0;
\]
\item for $a,a'\in\{0,\ldots,n,\infty\}$, $i,i'\in\{1,\ldots,\varepsilon_3(M)\}$, with $(a,i)\ne (a',i')$, we have
\[
F'_{a,i}\cdot F'_{a',i'}=\begin{cases}
1, & \text{if }(a,a')\in\{(0,n),(n,0)\}\text{ and } i=i',\\
0, & \text{otherwise}.
\end{cases}
\]
\end{itemize}
By \cite[Proposition~1.21 (a), Chapter 9]{Liu06}, for each prime divisor $C$ we have $C\cdot \mathcal X'_p=0$. Hence, from the previous intersection numbers we get the following self-intersection numbers:
\begin{align*}
(C'_a)^2&=\begin{cases}
-\frac{1}{12}d(M)p^{n-1}(p-1)-\frac{1}{2}\xi(-1)\varepsilon_2(M)-\frac{3-(-1)^n}{6}\xi(-3)\varepsilon_3(M), & \text{if }a\in\{0,n\}, \\
-\frac{1}{6}d(M)p^{|n-2a|}-\frac{1}{2}\varepsilon_2(M)-\frac{1}{3}\varepsilon_3(M)-\frac{1-(-1)^n}{6}\xi(-3)\varepsilon_3(M), & \text{if }a\in\{1,\ldots,n-1\},
\end{cases} \\
(E'_{a,i})^2&=-2, \quad \text{for }a\in\{1,\ldots,n-1,\infty\},\,\,i\in\{1,\ldots,\varepsilon_2(M)\}, \\
(F'_{a,i})^2&=\begin{cases}
-2, & \text{if }a\in\{0,n\}, \\
-3, & \text{if }a\in\{1,\ldots,n-1,\infty\},
\end{cases} \quad \text{for }i\in\{1,\ldots,\varepsilon_3(M)\}.
\end{align*}
 For the reader's convenience we show how to obtain for instance $(E'_{a,i})^2$, for $p\equiv 1 \pmod{12}$:
\begin{align*}
&E'_{a,i}\cdot \mathcal X'_p=0, \\
&E'_{a,i}\cdot\left[ C'_0+C'_n+\sum_{a'=1}^{n-1}\phi(p^{\min(a',n-a')})\left(C'_a+\frac{1}{2}\sum_{i'=1}^{\varepsilon_2(M)}E'_{a',i'}+\frac{1}{3}\sum_{i'=1}^{\varepsilon_3(M)}F'_{a',i'}\right)\right]=0, \\
&E'_{a,i}\cdot C'_0+E'_{a,i}\cdot C'_n+\sum_{a'=1}^{n-1}\phi(p^{\min(a',n-a')})\left(E'_{a,i}\cdot C'_{a'}+\frac{1}{2}\sum_{i'=1}^{\varepsilon_2(M)}E'_{a,i}\cdot E'_{a',i'}+\frac{1}{3}\sum_{i'=1}^{\varepsilon_3(M)}E'_{a,i}\cdot F'_{a',i'}\right)=0,\\
&0+0+\phi(p^{\min(a,n-a)})\left(1+\frac{1}{2}(E'_{a,i})^2+\frac{1}{3}\cdot 0\right)=0, \\
&(E'_{a,i})^2=-2.
\end{align*}
\begin{remark}\label{rem:Cn/2}
When $n$ is even, we observe that
\[
(C_{n/2}')^2=-\frac{d(M)+3\varepsilon_2(M)+2\varepsilon_3(M)}{6},
\]
hence it does not depend on $p^n$, but only on $M$.
\end{remark}

\subsection{Non-minimal Edixhoven's models}\label{sec:non_min-edi}
 We recall (see for instance \cite[Chapter 9.3]{Liu06}) that the arithmetic surface $\mathcal X'$ is said minimal if it doesn't contain exceptional divisors, i.e., it doesn't contain any vertical prime divisor (lying over a prime $\ell$) such that:
\begin{itemize}
\item $C\cong\mathbb P^1(\kappa)$, where $\kappa=H^0(C,\mathcal O_C)$;
\item $C^2=-[\kappa:\mathbb F_\ell]$.
\end{itemize}
Thus, the exceptional divisors of $\mathcal X'$, if any, are supported in the special fibres of $\mathcal X'$. Let $C$ be any component of the special fibre $\mathcal X'_p$. By \cite[Tag 0366]{stacks}, in order to show that $H^0(C,\mathcal O_C)=\mathbb F_p$ it is enough to verify that $C$ is geometrically reduced and geometrically connected. By \cite{Edi90}, we know that the base change of $C$ to $\overline {\mathbb F}_p$ is isomorphic to $\mathbb P^1(\overline{\mathbb F}_p)$, which is indeed reduced and connected. It follows that, in order to locate the exceptional divisors, we have to look exclusively at the self-intersections computed at the end of \Cref{sec:Edi}. They are always different from $-1$ except for:
\begin{itemize}
\item  $n=M=1$ and $p\in\{5,7,13\}$, in these cases $(C_0')^2=(C_1')^2=-1$;
\item  $M=1$ and $n$ even (it follows by \Cref{rem:Cn/2}), in these cases $(C_{n/2}')^2=-1$.
\end{itemize}
So in all these cases the Edixhoven's model is not minimal. 
Since we are interested in the asymptotic behaviour of the intersection numbers, we don't discuss the \emph{ad hoc} blow downs necessary when $N\in\{5,7,13\}$ and for simplicity we make the following assumption:
\begin{center}
\framebox{From now on and until the end of \Cref{sec:minmod} we fix $N>1$, $(N,6)=1$ and $N\notin\{ 5,7,13\}$.}
\end{center}

\noindent If $M=1$ and $n$ is even, we need three blow downs to get the minimal regular model. We denote the composition of these three blow downs by $\pi\colon\mathcal X'\to \mathcal X$. If $\mathcal X'$ is already minimal, $\pi$ is just the identity map. Let $C$ be a prime divisor of $\mathcal X$, let $\pi^*(C)$ be the pullback of $C$, let $C'$ be the corresponding divisor of $C$ in $\mathcal X'$ and, for $i=1,\ldots,c$, let $D_i$ be the $c$ prime divisors contracted by $\pi$. Repeatedly applying \cite[Proposition~2.23, Chapter 9]{Liu06}, we get\[
\pi^*(C)=C'+\sum_{i=1}^c d_iD_i,
\]
with $d_1,\ldots,d_c\in\ZZ$. So, we can compute the $d_i$'s using the fact that $D_i\cdot \pi^*(C)=0$, for every $i=1,\ldots,c$, see \cite[Theorem~2.12 (a), Chapter 9]{Liu06}. Now we list the results of the pullbacks when $M=1$ and $n$ is even. Since $d(1)=\varepsilon_2(1)=\varepsilon_3(1)=1$ (see \Cref{rem:genus} and \Cref{eq:k} for the notation), we have:
\begin{description}
\item[Case $p\equiv 1\pmod{12}$.] We have to contract $C'_{n/2}$, $E'_{n/2,1}$ and $F'_{n/2,1}$ obtaining
\begin{align*}
\pi^*(C_a)&=C'_a+6kC'_{n/2}+3kE'_{n/2,1}+2kF'_{n/2,1},\quad \text{for }a\ne n/2, \\
\pi^*(E_{a,1})&=E'_{a,1},\quad \text{for }a\ne n/2, \\
\pi^*(F_{a,1})&=F'_{a,1},\quad \text{for }a\ne n/2.
\end{align*}
 \item[Case $p\equiv 5\pmod{12}$.] We have to contract $C'_{n/2}$, $E'_{n/2,1}$ and $F_\infty'$ obtaining
\begin{align*}
\pi^*(C_a)&=C'_a+(6k+2)C'_{n/2}+(3k+1)E'_{n/2,1}+(2k+1)F'_{\infty,1},\quad \text{for }a\ne n/2, \\
\pi^*(E_{a,1})&=E'_{a,1},\quad \text{for }a\ne n/2.
\end{align*}
\item[Case $p\equiv 7\pmod{12}$.] We have to contract $C'_{n/2}$, $E'_{\infty,1}$ and $F'_{n/2,1}$ obtaining
\begin{align*}
\pi^*(C_a)&=C'_a+(6k+3)C'_{n/2}+(3k+2)E'_{\infty,1}+(2k+1)F'_{n/2,1},\quad \text{for }a\ne n/2, \\
\pi^*(F_{a,1})&=F'_{a,1},\quad \text{for }a\ne n/2.
\end{align*}
\item[Case $p\equiv 11\pmod{12}$.] We have to contract $C'_{n/2}$, $E'_{\infty,1}$ and $F'_{\infty,1}$ obtaining
\begin{align*}
\pi^*(C_a)&=C'_a+(6k+5)C'_{n/2}+(3k+3)E'_{\infty,1}+(2k+2)F'_{\infty,1},\quad \text{for }a\ne n/2.
\end{align*}
\end{description}

\subsection{Minimal regular models and intersection matrices}\label{sec:min}

Here we describe the minimal regular model for $X_0(N)$ with $N=p^n M>1$, coprime to 6 and $N\notin\{5,7,13\}$. As explained in \Cref{sec:non_min-edi} above, if either $M>1$ or $n$ is odd, the Edixhoven's model is already a minimal regular model. In every case we use the same notation as \Cref{sec:Edi} for model and components just dropping the dash when we refer to the minimal regular model. In the cases where blow downs are necessary, $N=p^n$ and the only fiber that change is $\mathcal X_p$, hence we write it below for convenience. Since $C\cdot D=\pi^*(C)\cdot\pi^*(D)$ for every prime divisor $C$ and $D$ (see \cite[Theorem~2.12 (c), Chapter 9]{Liu06}), using \Cref{sec:non_min-edi} we compute all the intersection numbers of the components of $\mathcal X_p$. Finally, we compute the self-intersection numbers, as in \Cref{sec:Edi} above, using the intersection numbers and the fact that for each prime divisor $C$ we have $C\cdot \mathcal X_p=0$ (\cite[Proposition~1.21 (a), Chapter~9]{Liu06}). Hence we get:
\allowdisplaybreaks
\begin{align*}
\mathcal X_p&=\sum_{\substack{a=0\\a\ne n/2}}^{n}\phi(p^{\min(a,n-a)})C_a+\sum_{\substack{a=1\\a\ne n/2}}^{n-1}\phi(p^{\min(a,n-a)})\left(\frac{1}{2}\xi(-3)E_{a,1}+\frac{1}{3}\xi(-1)F_{a,1}\right), \\
C_a\cdot C_{a'}&=C'_a\cdot C'_{a'}+6k^2+(6k+2)\xi(-1)+(4k+1)\xi(-3)+2\xi(-1)\xi(-3),\quad\text{for }a,a'\in\{0,\ldots,n\}-\{n/2\}, \\
C_a\cdot E_{a',1}&=\begin{cases}
1, & \text{if }a=a', \\
0, & \text{if }a\ne a',
\end{cases}\quad \text{for }a\in\{0,\ldots,n\}-\{n/2\},\,\,a'\in\{1,\ldots,n-1\}-\{n/2\}, \\
C_a\cdot F_{a',1}&=\begin{cases}
1, &\text{if }a=a', \\
0, & \text{if }a\ne a',
\end{cases}\quad \text{for }a\in\{0,\ldots,n\}-\{n/2\},\,\,a'\in\{1,\ldots,n-1\}-\{n/2\}, \\
E_{a,1}\cdot E_{a',1}&=\begin{cases}
-2, &\text{if }a=a', \\
0, & \text{if }a\ne a',
\end{cases} \quad \text{for }a,a'\in\{1,\ldots,n-1\}-\{n/2\}, \\
E_{a,1}\cdot F_{a',1}&=0, \quad \text{for }a,a'\in\{1,\ldots,n-1\}-\{n/2\}, \\
F_{a,1}\cdot F_{a',1}&=\begin{cases}
-3, &\text{if }a=a', \\
0, & \text{if }a\ne a',
\end{cases}\quad\text{for }a,a'\in\{1,\ldots,n-1\}-\{n/2\},
\end{align*}
\allowdisplaybreaks[0]
here $\phi$ is the Euler’s totient function, $\xi$ is defined in \Cref{eq:xi} and $k$ is defined in \Cref{eq:k}.

\section{Asymptotics for the self-intersection of $\overline\omega$}\label{sec:main_res}

The proof of \Cref{thm:main_res} is performed in several steps. We continue with the same notation as the previous sections, so $f\colon \mathcal X\to\operatorname{Spec} \mathbb Z$ is the minimal regular model of $X_0(N)$. The genus of $X_0(N)$ is $g$ and for each prime $\ell\nmid N$ the fiber $\mathcal X_\ell$ is irreducible with genus $g$. It is well known that the genus of $X_0(N)$ is $0$ if and only if $N\in\{1,2,3,4,5,6,7,8,9,10,12,13,16,18,25\}$, so in these cases the canonical K\"ahler form $\Omega^{\text{can}}$ is not well defined, therefore we make the following assumption:

\begin{center}
\framebox{From now on  we fix $N>1$, $(N,6)=1$ and $N\notin\{ 5,7,13,25\}$.}
\end{center}

\noindent In  \Cref{sub:aformula} we prove the following formula for the self-intersection of the canonical Arakelov divisor:
\begin{equation}\label{eq:mainformula}
\langle\overline{\omega}, \overline{\omega}\rangle=\underbrace{-4g(g-1)\langle H_0, H_\infty\rangle}_{(a)}\underbrace{+\frac{g\langle V_0, V_\infty\rangle}{g-1}-\frac{\langle V_0, V_0\rangle+\langle V_\infty, V_\infty\rangle}{2g-2}}_{(b)}\underbrace{+\frac{h_0+h_\infty}{2}}_{(c)},
\end{equation}
where $H_0$ and $H_\infty$ are the horizontal divisors on $\mathcal X$ induced by the cusps $0$ and $\infty$ of $X_0(N)$, $V_0$ and $V_\infty$ are two carefully chosen vertical divisors supported over the primes dividing $N$, $h_0$ and $h_\infty$ are certain Néron-Tate's heights of some points in the Jacobian  of $X_0(N)$. We stress that the crucial point is the \emph{constructive proof} of the existence of $V_0$ and $V_\infty$ (see \Cref{prop:V_i}). In fact assume that $\mathcal P$ is the set of primes dividing $N$ and that $V^{(p)}_m$ is the part of $V_m$ supported over $p$, for $p\in\mathcal P$ and $m\in\{0,\infty\}$, then  the vectors made of the  (rational) multiplicities of the components of $V^{(p)}_0$ and $V^{(p)}_\infty$ are the solutions of two linear systems. These systems are described in \Cref{the_system} in \Cref{sub:aformula} and solved in \Cref{prop:general} in \Cref{sec:comp}. The last step of the proof consists in computing the asymptotics for $N\to+\infty$ for all summands of \Cref{eq:mainformula}. The estimates of the pieces $(a)$ and $(c)$ of the formula are already known in the literature, so in \Cref{sec:concl} we \emph{exactly compute} $(b)$ and we study its asymptotics. Finally,  we put all together to conclude the proof.

\subsection{A formula for $\langle\overline{\omega},\overline{\omega}\rangle$}\label{sub:aformula}

The results of this section hold for every regular model (not necessarily minimal), but for simplicity we refer to the minimal regular model $\mathcal X$. We denote by $H_0$ and $H_\infty$ the closures in $\mathcal X$ of the two cusps $0$ and $\infty$. By \cite[Proposition~1.30, Chapter~9]{Liu06}, we know that $ H_m\cdot \mathcal X_q=1$, for $m\in\{0,\infty\}$, and every prime $q$. Moreover, by the components labelling of \cite[page 296]{KM85}, we can assume that $H_0$  and $H_\infty$ respectively meet transversally the components $C_0$ and $C_n$ of each special fiber. We now define two particular divisors $G_0$ and $G_\infty$ and we describe some their properties that are important to write an explicit formula for the self-intersection. Let $\mathcal K$ be a canonical divisor of $\mathcal X$, then we define:
\begin{equation}\label{eq:Gm}
G_m:=\mathcal K-(2g-2)H_m,\quad \text{for $m\in\{0,\infty\}$}.
\end{equation}

\begin{proposition}\label{prop:deg0}
Let $F$ be a  vertical divisor  of $\mathcal X$ which is not supported over the primes diving $N$, then
\[
 G_m\cdot F=0, \quad \text{for  $m\in\{0,\infty\}$}.
\]
\end{proposition}
\begin{proof}
Since the primes dividing $N$ are the only primes whose fiber can contain more than one component, we write $F=\sum_{\ell\nmid N} n_\ell\mathcal X_\ell$, where $n_\ell\in\mathbb Z$. By  \cite[Proposition~1.35, Chapter~9]{Liu06}, we have $\mathcal K\cdot\mathcal X_\ell=2g-2$. Moreover, by \cite[Proposition~1.30, Chapter~9]{Liu06}, we have $H_m\cdot \mathcal X_\ell=1$.
\end{proof}

\begin{definition}
A divisor $D$ of $\mathcal X$ is called \emph{$f$-numerically trivial} if $D\cdot F=0$ for every vertical divisor $F$.
\end{definition}
In the following proposition we explain how we can modify the divisors $G_m$ in order to get $f$-numerically trivial $\mathbb Q$-divisors. It is a special case of  \cite[Lemma~2.2.2]{Mor13}, but we prefer write a full proof since it is constructive and it is crucial in order to carry out the computations of \Cref{sec:comp}.

\begin{proposition}\label{prop:V_i}
There are two vertical $\mathbb Q$-divisors $V_0$ and $V_\infty$ supported over the primes dividing $N$ such that 
\[
D_m:=G_m+V_m,\quad \text{for }m\in\{0,\infty\},
\]
is $f$-numerically trivial.
\end{proposition}
\begin{proof}
We explain how to effectively construct infinitely many $V_0$ and $V_\infty$ that satisfy the proposition. Let $\mathcal P$ be the set of primes dividing $N$ and let $\nu_p$ be the number of distinct components of $X_{p}$. For simplicity of presentation we relabel all the components and the relative multiplicities so that we can write:
\[
\mathcal X_{p}=\sum^{\nu_p}_{i=1} m^{(p)}_i\Gamma^{(p)}_i,\quad \text{ for } p\in\mathcal P.
\]
Then by using \cite[Remark 1.31, Chapter 9]{Liu06} we get:
\[
\sum^{\nu_p}_{i=1} m^{(p)}_i(G_m\cdot\Gamma^{(p)}_i)= G_m\cdot\mathcal X_{p}=\deg_{\mathbb Q}G_{m,\mathbb Q}=0,
\]
where $G_{m,\mathbb Q}$ is the restriction of $G_m$ on the generic fibre. By \cite[Lemma~2.2.1]{Mor13}, it follows that the vector $(G_m\cdot\Gamma^{(p)}_1,\ldots, G_m\cdot\Gamma^{(p)}_{\nu_p})\in\mathbb Q^{\nu_p}$ lies in the image of the linear function induced by the matrix $(\Gamma^{(p)}_i\cdot\Gamma^{(p)}_j)_{ij}$. So we can find a solution  for the linear system:
\begin{equation}\label{the_system}
\begin{pmatrix}
\Gamma^{(p)}_1\cdot \Gamma^{(p)}_1 & \dots &  \Gamma^{(p)}_1\cdot \Gamma^{(p)}_{\nu_p}\\
\vdots    &  & \vdots\\
\Gamma^{(p)}_{\nu_p}\cdot \Gamma^{(p)}_1 & \dots & \Gamma^{(p)}_{\nu_p}\cdot \Gamma^{(p)}_{\nu_p}
\end{pmatrix}
\begin{pmatrix}
x^{(p)}_1 \\ \vdots \\ x^{(p)}_{\nu_p} 
\end{pmatrix}
=
\begin{pmatrix}
-G_m\cdot \Gamma^{(p)}_1 \\ \vdots \\ -G_m\cdot \Gamma^{(p)}_{\nu_p} 
\end{pmatrix}.
\end{equation}
Actually \cite[Lemma~2.2.1]{Mor13} shows that the dimension of the affine subspace of the solutions of system (\ref{the_system}) is one and a basis of its direction is given by $(m^{(p)}_1,\ldots,m^{(p)}_{\nu_p})$. If $(x^{(p)}_1,\ldots x^{(p)}_{\nu_p})$ is any solution, we put:
\begin{equation}\label{eq:Vm}
V^{(p)}_m:=\sum^{\nu_p}_{i=1} x^{(p)}_i\Gamma^{(p)}_i,\quad \text{for } m\in\{0,\infty\}.
\end{equation}
We define
\[
V_m:=\sum_{p\in\mathcal P} V^{(p)}_m,\quad \text{for } m\in\{0,\infty\}.
\]
By \Cref{prop:deg0} and the fact that $V_m$ is supported over $\mathcal P$, it follows that $(G_m+V_m)\cdot \mathcal X_\ell=0$, for every $\ell\nmid N$. Moreover, by the construction of the vector $(x^{(p)}_1,\ldots, x^{(p)}_{\nu_p})$ for $p\in\mathcal P$, we have that $(G_m+V_m)\cdot F=0$ for every divisor $F$ supported over~$\mathcal P$.
\end{proof}
From now on we fix a choice of $V_0$ and $V_\infty$ that satisfy   \Cref{prop:V_i}, so the divisors $D_0$ and $D_\infty$ are fixed as well. We recall that $\langle\cdot,\cdot\rangle$ denotes the Arakelov intersection pairing. If one of the two Arakelov divisors is an ordinary vertical divisor supported over a prime $p$, then their Arakelov intersection differs from the usual intersection pairing (denoted by a dot) by a $\log p$ factor. We define the numbers:
\[
h_m:=\langle D_m, D_m\rangle, \quad\text{for $m\in\{0,\infty\}$}.
\]
In the next proposition we show how the self-intersection of the Arakelov canonical divisor can be expressed in terms of  $h_m$ and some intersection numbers involving only $H_m$ and $V_m$.

\begin{proposition}\label{prop:omega2}
With the notations fixed above and for any $V_0,V_\infty$ satisfying  \Cref{prop:V_i} we have that:
\[
\langle\overline{\omega}, \overline{\omega}\rangle=-4g(g-1)\langle H_0, H_\infty\rangle+\frac{g\langle V_0, V_\infty\rangle}{g-1}-\frac{\langle V_0, V_0\rangle+\langle V_\infty, V_\infty\rangle}{2g-2}+\frac{h_0+h_\infty}{2}.
\]
\end{proposition}
\begin{proof}
As explained in \Cref{sec:arakelov} we know that $\langle\overline{\omega}, \overline{\omega}\rangle=\langle \mathcal K, \mathcal K\rangle$. By \cite[Theorem 4(c)]{Fal84}, we have that $h_m=-2\widehat h(D_{m,\mathbb Q})$ where  $D_{m,\mathbb Q}$ can be seen as a point in the Jacobian of $X_0(N)$ and $\widehat{h}$ is the Néron-Tate's height. But $\langle D_m, V_m\rangle=0$ by \Cref{prop:V_i}, so:
 \begin{equation}\label{eq:KK1}
 h_m=\langle D_m, D_m-V_m\rangle=\langle\mathcal K-(2g-2)H_m+V_m, \mathcal K-(2g-2)H_m \rangle.
 \end{equation}
 Expanding \Cref{eq:KK1} we obtain:
 \begin{equation}\label{eq:KK2}
  \langle \mathcal K, \mathcal K\rangle= 2(2g-2)\langle \mathcal K, H_m\rangle -(2g-2)^2\langle H_m, H_m\rangle-\langle V_m, \mathcal K\rangle+(2g-2)\langle V_m, H_m\rangle+h_m.
 \end{equation}
We now expand $\langle D_m, V_m\rangle=0$ to get the equality
\begin{equation}\label{subsKK2}
-\langle V_m, \mathcal K\rangle+(2g-2)\langle V_m, H_m\rangle=\langle V_m, V_m\rangle.
\end{equation}
Moreover the Arakelov adjunction formula (see \cite[Corollary 5.6]{Lan88}) says that
\begin{equation}\label{subssKK2}
\langle \mathcal K, H_m\rangle=-\langle H_m, H_m\rangle.
\end{equation}
By substituting Equations (\ref{subsKK2}) and (\ref{subssKK2}) inside \Cref{eq:KK2} we get:
\begin{equation*}\label{KK3}
   \langle \mathcal K, \mathcal K\rangle= -4g(g-1)\langle H_m, H_m\rangle+\langle V_m, V_m\rangle+h_m.
\end{equation*}
Now summing for $m\in\{0,\infty\}$, gives:
\begin{equation}\label{eq:KK4}
   \langle \mathcal K, \mathcal K\rangle= -2g(g-1)\big(\langle H_0, H_0\rangle+\langle H_\infty, H_\infty\rangle\big)+\frac{\langle V_0, V_0\rangle+\langle V_\infty, V_\infty\rangle+h_0+h_\infty}{2}.
\end{equation}
Consider the divisor
\[
D_{\infty}-D_0=(2g-2)(H_0-H_{\infty})+V_{\infty}-V_0.
\]
It satisfies the hypotheses of \cite[Theorem 4(c)]{Fal84}; but  $(D_{\infty}-D_0)_{\mathbb Q}$ is supported on the cusps of $X_0(N)$, therefore by the Manin-Drinfeld's theorem (see \cite[Corollary~3.6]{Man72} and \cite[Theorem~1]{Dri73}) it is a torsion element in the Jacobian. It follows that $\langle D_{\infty}-D_0, D_{\infty}-D_0\rangle=0$, which means:
\begin{equation}\label{eq:KK5}
\langle H_0, H_0\rangle+\langle H_\infty, H_\infty\rangle=2\langle H_0, H_\infty\rangle+\frac{\langle V_0, V_0\rangle+\langle V_\infty, V_\infty\rangle-2\langle V_0, V_\infty\rangle}{(2g-2)^2}.
\end{equation}
Substituting \Cref{eq:KK5} inside \Cref{eq:KK4} we finally get:
\[
   \langle \mathcal K, \mathcal K\rangle=-4g(g-1)\langle H_0, H_\infty\rangle-\frac{\langle V_0, V_0\rangle+\langle V_\infty, V_\infty\rangle}{2g-2}+\frac{g\langle V_0, V_\infty\rangle}{g-1}+\frac{h_0+h_\infty}{2}.
\]
\end{proof}

We point out that the strategy we used to obtain the above formula for  $\langle\overline{\omega}, \overline{\omega}\rangle$ cannot be applied to the minimal model of a general algebraic curve. In fact the proof of \Cref{prop:omega2} depends heavily on the Manin-Drinfeld's theorem which is a specific property of modular curves.

\subsection{Computation of $V_0$ and $V_\infty$}\label{sec:comp}

\begin{center}
\framebox{In this subsection we fix $n$ odd or $M>1$.}
\end{center}

\noindent Let $\mathcal K$ be a canonical divisor of $\mathcal X$, let $C$ be one of the components of the special fiber $\mathcal X_p$ in $\mathcal X$ and let $g_C$ be the genus of $C$. By the adjunction formula we know that $C\cdot(C+\mathcal K)=2g_C-2$ (see for instance \cite[Theorem~1.37, Chapter~9]{Liu06}). From this, we can compute the intersection $C\cdot \mathcal K$ because:
\begin{align}\label{eq:CK}
C\cdot(C+\mathcal K)&=2g_{C}-2, \nonumber \\
C\cdot C+C\cdot\mathcal K&=2g_{C}-2, \nonumber \\
C\cdot\mathcal K&=2g_{C}-2-C^2.
\end{align}
By \Cref{eq:g=0,eq:Kcap} and using the intersection numbers computed in \Cref{sec:Edi,sec:min}, now we can explicitly write down the linear systems in \Cref{the_system} for every $p\mid N$ and $m\in\{0,\infty\}$. We remark that the constant terms of the systems
\[
-G_m\cdot C=(2g-2)H_m\cdot C-\mathcal K\cdot C=(2g-2)H_m\cdot C-2g_{C}+2+C^2, \quad\text{for }C \text{ a component},
\]
are given by \Cref{eq:Gm,eq:CK} and
\begin{align*}
H_m\cdot C_0&=\begin{cases}
1, & \text{if } m=0, \\
0, & \text{if } m=\infty,
\end{cases} \\
H_m\cdot C_n&=\begin{cases}
0, & \text{if } m=0, \\
1, & \text{if } m=\infty,
\end{cases} \\
H_m\cdot C_a&=H_m\cdot E_{a,i}=H_m\cdot F_{a,i}=0,\quad \text{for every }m,a,i.
\end{align*}

In the following proposition we describe the solutions of the linear systems in \Cref{the_system}.
\begin{proposition}\label{prop:general}
Let $w=(w_{\Gamma^{(p)}_i})_{i=1,\ldots,\nu_p}$ be a generator of the kernel of the intersection matrix in \Cref{the_system}, let $u=(u_{\Gamma^{(p)}_i})_{i=1,\ldots,\nu_p}$ be a particular solution of the linear system in \Cref{the_system} with $m=0$, let $v=(v_{\Gamma^{(p)}_i})_{i=1,\ldots,\nu_p}$ be a particular solution of the linear system in \Cref{the_system} with $m=\infty$, let $\phi$ be the Euler's totient function  and let $g,d,\varepsilon_\infty,\varepsilon_2,\varepsilon_3$ be defined as in \Cref{rem:genus}. Then
\allowdisplaybreaks
\begin{align*}
w_{C_a}&=\phi(p^{\min(a,n-a)}), \quad\text{for }a\in\{0,\ldots,n\}, \\
w_{E_{a,i}}&=\begin{cases}
\frac{1}{2}\phi(p^{\min(a,n-a)}), & \text{if }a\in\{1,\ldots,n-1\}, \\
\frac{1}{2}\varepsilon_\infty(p^n), & \text{if }a=\infty,
\end{cases} \quad\text{for }i\in\{1,\ldots,\varepsilon_2(M)\}, \\
w_{F_{a,i}}&=\begin{cases}
\frac{1}{3}\phi(p^{\min(a,n-a)}), & \text{if }a\in\{1,\ldots,n-1\}, \\
\frac{1}{3}\varepsilon_\infty(p^n), & \text{if }a=\infty, \\
\frac{1}{2}\varepsilon_\infty(p^n), & \text{if }a\in\{0,n\},
\end{cases} \quad\text{for }i\in\{1,\ldots,\varepsilon_3(M)\}, \\
u_{C_0}&=0, \\
u_{C_a}&=\frac{2(g-1)\phi\left(p^{\min(a,n-a)}\right)}{d(N)(p-1)}\left(6a(p-1)+6+\frac{d(N)-3(\min(a,n-a)(p-1)+1)\varepsilon_\infty(N)}{g-1}\right), \\
& \text{for }a\in\{1,\ldots,n-1\}, \\
u_{C_n}&=\frac{2(g-1)}{d(N)(p-1)}\left(6n(p-1)+12\right), \\
u_{E_{a,i}}&=\frac{(g-1)\phi\left(p^{\min(a,n-a)}\right)}{d(N)(p-1)}\left(6a(p-1)+6+\frac{d(N)-3(\min(a,n-a)(p-1)+1)\varepsilon_\infty(N)}{g-1}\right), \\
& \text{for }a\in\{1,\ldots,n-1\},i\in\{1,\ldots,\varepsilon_2(M)\}, \\
u_{E_{\infty,i}}&=\frac{(g-1)\varepsilon_\infty(p^{n})}{2d(N)(p-1)}\left(6n(p-1)+36+\frac{-3(n(p-1)-2)\varepsilon_\infty(N)+8\varepsilon_3(N)}{g-1}\right)-\frac{2}{p-1}+\frac{6\varepsilon_\infty(M)}{d(M)(p-1)}, \\
& \text{for }i\in\{1,\ldots,\varepsilon_2(M)\}, \\
u_{F_{a,i}}&=\frac{1}{3}+\frac{2(g-1)\phi\left(p^{\min(a,n-a)}\right)}{3d(N)(p-1)}\left(6a(p-1)+6+\frac{d(N)-3(\min(a,n-a)(p-1)+1)\varepsilon_\infty(N)}{g-1}\right), \\
& \text{for }a\in\{1,\ldots,n-1\},i\in\{1,\ldots,\varepsilon_3(M)\}, \\
u_{F_{\infty,i}}&=\frac{1}{3}+\frac{(g-1)\varepsilon_\infty(p^{n})}{d(N)(p-1)}\Big(2n(p-1)+12+\frac{-(n(p-1)-2)\varepsilon_\infty(N)+2\varepsilon_2(N)}{g-1}\Big)-\frac{4}{3(p-1)}+\frac{4\varepsilon_\infty(M)}{d(M)(p-1)}, \\
& \text{for }i\in\{1,\ldots,\varepsilon_3(M)\}, \\
u_{F_{0,i}}&=\frac{(g-1)\varepsilon_\infty(p^{n})}{2d(N)(p-1)}\left(6n(p-1)-2(p+1)+36+\frac{-3(n(p-1)-2)\varepsilon_\infty(N)+6\varepsilon_2(N)}{g-1}\right)-\frac{2}{p-1}+\frac{6\varepsilon_\infty(M)}{d(M)(p-1)}, \\
& \text{for }i\in\{1,\ldots,\varepsilon_3(M)\}, \\
u_{F_{n,i}}&=\frac{(g-1)\varepsilon_\infty(p^{n})}{2d(N)(p-1)}\left(6n(p-1)+2(p+1)+36+\frac{-3(n(p-1)-2)\varepsilon_\infty(N)+6\varepsilon_2(N)}{g-1}\right)-\frac{2}{p-1}+\frac{6\varepsilon_\infty(M)}{d(M)(p-1)}, \\
& \text{for }i\in\{1,\ldots,\varepsilon_3(M)\}, \\
v_{C_0}&=0, \\
v_{C_a}&=\frac{2(g-1)\phi\left(p^{\min(a,n-a)}\right)}{d(N)(p-1)}\left(-6a(p-1)-6+\frac{d(N)-3(\min(a,n-a)(p-1)+1)\varepsilon_\infty(N)}{g-1}\right), \\
& \text{for }a\in\{1,\ldots,n-1\}, \\
v_{C_n}&=\frac{2(g-1)}{d(N)(p-1)}\left(-6n(p-1)-12\right), \\
v_{E_{a,i}}&=\frac{(g-1)\phi\left(p^{\min(a,n-a)}\right)}{d(N)(p-1)}\left(-6a(p-1)-6+\frac{d(N)-3(\min(a,n-a)(p-1)+1)\varepsilon_\infty(N)}{g-1}\right), \\
& \text{for }a\in\{1,\ldots,n-1\},i\in\{1,\ldots,\varepsilon_2(M)\}, \\
v_{E_{\infty,i}}&=\frac{(g-1)\varepsilon_\infty(p^{n})}{2d(N)(p-1)}\Big(-6n(p-1)+12+\frac{-3(n(p-1)-2)\varepsilon_\infty(N)+8\varepsilon_3(N)}{g-1}\Big)-\frac{2}{p-1}+\frac{6\varepsilon_\infty(M)}{d(M)(p-1)}, \\
& \text{for }i\in\{1,\ldots,\varepsilon_2(M)\}, \\
v_{F_{a,i}}&=\frac{1}{3}+\frac{2(g-1)\phi\left(p^{\min(a,n-a)}\right)}{3d(N)(p-1)}\left(-6a(p-1)-6+\frac{d(N)-3(\min(a,n-a)(p-1)+1)\varepsilon_\infty(N)}{g-1}\right), \\
& \text{for }a\in\{1,\ldots,n-1\},i\in\{1,\ldots,\varepsilon_3(M)\}, \\
v_{F_{\infty,i}}&=\frac{1}{3}+\frac{(g-1)\varepsilon_\infty(p^{n})}{d(N)(p-1)}\Big(-2n(p-1)+4+\frac{-(n(p-1)-2)\varepsilon_\infty(N)+2\varepsilon_2(N)}{g-1}\Big)-\frac{4}{3(p-1)}+\frac{4\varepsilon_\infty(M)}{d(M)(p-1)}, \\
& \text{for }i\in\{1,\ldots,\varepsilon_3(M)\}, \\
v_{F_{0,i}}&=\frac{(g-1)\varepsilon_\infty(p^{n})}{2d(N)(p-1)}\left(-6n(p-1)+2(p+1)+12+\frac{-3(n(p-1)-2)\varepsilon_\infty(N)+6\varepsilon_2(N)}{g-1}\right)-\frac{2}{p-1}+\frac{6\varepsilon_\infty(M)}{d(M)(p-1)}, \\
& \text{for }i\in\{1,\ldots,\varepsilon_3(M)\}, \\
v_{F_{n,i}}&=\frac{(g-1)\varepsilon_\infty(p^{n})}{2d(N)(p-1)}\left(-6n(p-1)-2(p+1)+12+\frac{-3(n(p-1)-2)\varepsilon_\infty(N)+6\varepsilon_2(N)}{g-1}\right)-\frac{2}{p-1}+\frac{6\varepsilon_\infty(M)}{d(M)(p-1)}, \\
& \text{for }i\in\{1,\ldots,\varepsilon_3(M)\}.
\end{align*}
\allowdisplaybreaks[0]
\end{proposition}
\begin{proof}
We checked these solutions using the software Mathematica. See \cite{codes} for the details.
\end{proof}
\begin{corollary}\label{cor:general}
Let $\mathcal P$ be the set of prime dividing $N$. The divisors $V_0$ and $V_\infty$ defined in \Cref{eq:Vm} can be chosen as follows
\begin{align*}
V_0 =\sum_{p\in\mathcal P}V_0^{(p)}=\sum_{p\in\mathcal P}\sum^{\nu_p}_{i=1} u_{\Gamma^{(p)}_i}\Gamma^{(p)}_i, \quad\text{and}\quad V_\infty =\sum_{p\in\mathcal P}V_\infty^{(p)}=\sum_{p\in\mathcal P}\sum^{\nu_p}_{i=1} v_{\Gamma^{(p)}_i}\Gamma^{(p)}_i,
\end{align*}
where $u_{\Gamma^{(p)}_i}$ and $v_{\Gamma^{(p)}_i}$ are described in \Cref{prop:general} above.
\end{corollary}

\subsection{$V_0$ and $V_\infty$ in the special case $n$ even and $M=1$}\label{sec:thespecialone}

If $n$ is even and $M=1$, before solving the linear systems in \Cref{the_system} that correspond to the minimal model $\mathcal X$, we have to blow down by $\pi$ the Edixhoven's model $\mathcal X'$ as explained in \Cref{sec:non_min-edi}. But we noticed that the solutions of the systems after the blow downs are exactly the same as those written in \Cref{prop:general} after removing the equations and the variables corresponding to the contracted components. The comparison has been performed by using the software Mathematica (see \cite{codes} for the details). In the remainder of this section we describe as the constant terms of the linear systems change after the blow downs. 

In this case the divisors $V_0$ and $V_\infty$ are supported only over $p$. Keeping the notation of \Cref{sec:minmod}, we denote by $C_a,E_{a,1},F_{a,1}$ the components of the special fiber $\mathcal X_p$ in $\mathcal X$ and by $C'_a,E'_{a,1},F'_{a,1}$ components of the special fiber $\mathcal X'_p$ in  $\mathcal X'$. Moreover $\mathcal K$ is a canonical divisor of $\mathcal X$ and $\mathcal K'$ is a canonical divisor of $\mathcal X'$. The pullback on $\mathcal X'$ of the canonical divisor $\mathcal K$ of $\mathcal X$ is
\[
\pi^*(\mathcal K)=\begin{cases}
\mathcal K'-4C'_{n/2}-2E'_{n/2,1}-F'_{n/2,1}, & \text{if }p\equiv 1 \pmod{12}, \\
\mathcal K'-4C'_{n/2}-2E'_{n/2,1}-F'_{\infty,1}, & \text{if }p\equiv 5 \pmod{12}, \\
\mathcal K'-4C'_{n/2}-2E'_{\infty,1}-F'_{n/2,1}, & \text{if }p\equiv 7 \pmod{12}, \\
\mathcal K'-4C'_{n/2}-2E'_{\infty,1}-F'_{\infty,1}, & \text{if }p\equiv 11 \pmod{12},
\end{cases}
\]
where the computations are done as explained in \Cref{sec:non_min-edi}. Since when $M=1$ we have that
\begin{equation}\label{eq:g=0}
g_{C'_a}=g_{E'_{a',1}}=g_{F'_{a',1}}=0, \quad \text{for every }a\in\{0,\ldots,n\}\text{ and }a'\in\{1,\ldots,n-1,\infty\},
\end{equation}
it follows that
\begin{equation}\label{eq:Kcap}
\begin{aligned}
C_a\cdot\mathcal K&=-(C'_a)^2-4k-2-2\xi(-1)-\xi(-3),\quad\text{for }a\in\{0,\ldots,n\}- \{n/2\}, \\ 
E_{a,1}\cdot\mathcal K&=0,\quad \text{for }a\in\{1,\ldots,n-1,\infty\}- \{n/2\}, \\ 
F_{a,1}\cdot\mathcal K&=1,\quad \text{for }a\in\{1,\ldots,n-1,\infty\}- \{n/2\},
\end{aligned}
\end{equation}
where $\xi$ is defined in \Cref{eq:xi} and $k$ is defined in \Cref{eq:k}. We explain how to compute, for instance, $C_a\cdot\mathcal K$ with $a\ne n/2$ and when $p\equiv 1\pmod{12}$:
\begin{align*}
C_a\cdot\mathcal K&=\pi^*(C_a)\cdot\pi^*(\mathcal K)= \\
&=(C'_a+6kC'_{n/2}+3kE'_{n/2,1}+2kF'_{n/2,1})\cdot(\mathcal K'-4C'_{n/2}-2E'_{n/2,1}-F'_{n/2,1})= \\
&=C'_a\cdot\mathcal K'-4C'_a\cdot C'_{n/2}+6kC'_{n/2}\cdot\mathcal K'-24k(C'_{n/2})^2-24kC'_{n/2}\cdot E'_{n/2,1}-14kC'_{n/2}\cdot F'_{n/2,1}+ \\
&\quad+3kE'_{n/2,1}\cdot\mathcal K'-6k(E'_{n/2,1})^2+2kF'_{n/2,1}\cdot\mathcal K'-2k(F'_{n/2,1})^2= \\
&=2g_{C'_a}-2-(C'_a)^2-4k+6k(2g_{C'_{n/2}}-2-(C'_{n/2})^2)+24k-24k-14k+\\
&\quad+3k(2g_{E'_{n/2,1}}-2-(E'_{n/2,1})^2)+12k+2k(2g_{F'_{n/2,1}}-2-(F'_{n/2,1})^2)+6k= \\
&=2g_{C'_a}-2-(C'_a)^2-4k-6k-14k+12k+2k+6k=-2-(C'_a)^2-4k.
\end{align*}
Hence, to obtain the constant terms of the linear systems in \Cref{the_system} for the model after the blow downs it is enough to substitute \Cref{eq:Kcap} in
\[
-G_m\cdot C=(2g-2)H_m\cdot C-\mathcal K\cdot C, \quad\text{for }C \text{ a component},
\]
and recall that
\begin{align*}
H_m\cdot C_0&=\begin{cases}
1, & \text{if } m=0, \\
0, & \text{if } m=\infty,
\end{cases} \\
H_m\cdot C_n&=\begin{cases}
0, & \text{if } m=0, \\
1, & \text{if } m=\infty,
\end{cases} \\
H_m\cdot C_a&=H_m\cdot E_{a,1}=H_m\cdot F_{a,1}=0,\quad \text{for every }m,a.
\end{align*}
\begin{remark}
We do not need this in the following, but we remark that by the previous computations and using \Cref{eq:CK}, we find that the blow downs change the genera of the prime divisors in the following way:
\begin{align*}
g_{C_a}&=g_{C'_a}+3k^2+(3\xi(-1)+2\xi(-3)-2)k+\xi(-1)\xi(-3),\quad\text{for }a\in\{0,\ldots,n\}- \{n/2\}, \\
g_{E_{a,1}}&=g_{E'_{a,1}},\quad \text{for }a\in\{1,\ldots,n-1,\infty\}- \{n/2\}, \\
g_{F_{a,1}}&=g_{F'_{a,1}},\quad \text{for }a\in\{1,\ldots,n-1,\infty\}- \{n/2\},
\end{align*}
where $\xi$ is defined in \Cref{eq:xi} and $k$ is defined in \Cref{eq:k}. For example, again for $C_a$ with $a\ne n/2$ and $p\equiv 1\pmod{12}$, we have that
\[
2g_{C_a}-2-C_a^2=C_a\cdot\mathcal K=2g_{C'_a}-2-(C'_a)^2-4k,
\]
and then
\[
g_{C_a}=g_{C'_a}+\frac{C_a^2-(C'_a)^2}{2}-2k=g_{C'_a}+3k^2-2k.
\]
Moreover, by \Cref{eq:g=0}, we have
\begin{align*}
g_{C_a}&=3k^2+(3\xi(-1)+2\xi(-3)-2)k+\xi(-1)\xi(-3),\quad\text{for }a\in\{0,\ldots,n\}- \{n/2\}, \\
g_{E_{a,1}}&=0,\quad \text{for }a\in\{1,\ldots,n-1,\infty\}- \{n/2\}, \\
g_{F_{a,1}}&=0,\quad \text{for }a\in\{1,\ldots,n-1,\infty\}- \{n/2\}.
\end{align*}
\end{remark}

\subsection{Conclusion of the proof}\label{sec:concl}

\begin{proof}[Proof of \Cref{thm:main_res}]
We now write down the asymptotics for the summands of \Cref{eq:mainformula} that we recall here:
\[
\langle\overline{\omega}, \overline{\omega}\rangle=\underbrace{-4g(g-1)\langle H_0, H_\infty\rangle}_{(a)}\underbrace{+\frac{g\langle V_0, V_\infty\rangle}{g-1}-\frac{\langle V_0, V_0\rangle+\langle V_\infty, V_\infty\rangle}{2g-2}}_{(b)}\underbrace{+\frac{h_0+h_\infty}{2}}_{(c)}.
\]
Regarding the summand $(a)$: \cite[Theorem 1.2]{MvP22} says that $2g(1-g)\mathcal G(0,\infty)=2g\log N+o(g\log N)$, so
\[
-4g(g-1)\langle H_0, H_\infty\rangle=2g\log N+o(g\log N), \quad \text{for }N\to+\infty,
\]
where we use the relation $2\langle H_0, H_\infty\rangle=-\mathcal G(0,\infty)$ explained in the last paragraph of \Cref{sec:arakelov}. In \cite[Section 6]{MU98} and \cite[Lemme 4.1.1]{AU97} it is shown that\footnote{Such result holds for every $N>1$, in fact the square-free hypothesis of \cite{MU98} and \cite{AU97} is never used in this point.}:
\begin{equation}\label{eq:h_m}
h_m=\begin{cases}
0, & \text{if }\varepsilon_2(N)=\varepsilon_3(N)=0,\\
O(\tau(N)^2\log N), & \text{otherwise}, \\
\end{cases}
\end{equation}
where $\tau(N)$ denotes the number of positive divisors of $N$.  In \cite{NR83}  it is proved that $\tau(N)=O(N^{\frac{1,538\log 2}{\log\log N}})$. Therefore regarding the summand $(c)$ we have:
\[
\frac{h_0+h_\infty}{2}=o(g\log N), \quad \text{for }N\to+\infty.
\]
Now we focus on summand (b). If $\mathcal P$ denotes the set of primes dividing $N$, we have
\[
\frac{g\langle V_0, V_\infty\rangle}{g-1}-\frac{\langle V_0, V_0\rangle+\langle V_\infty, V_\infty\rangle}{2g-2}=\sum_{p\in\mathcal P}\frac{g\langle V_0^{(p)}, V_\infty^{(p)}\rangle}{g-1}-\frac{\langle V_0^{(p)}, V_0^{(p)}\rangle+\langle V_\infty^{(p)}, V_\infty^{(p)}\rangle}{2g-2}.
\]
By \Cref{prop:general}, for every $p\in\mathcal P$, we get
\begin{equation}\label{eq:opicc}
\frac{g\langle V_0^{(p)}, V_\infty^{(p)}\rangle}{g-1}-\frac{\langle V_0^{(p)}, V_0^{(p)}\rangle+\langle V_\infty^{(p)}, V_\infty^{(p)}\rangle}{2g-2}=ng\log p+\frac{f(p,n,M)\log p}{12d(N)(p-1)},
\end{equation}
where
\begin{align*}
   &f(p,n,M):=2d(N)^2-8d(N)d(M)(p^{n-1}-1)-6(n(p-1)-4)d(N)\varepsilon_\infty(N)-96d(N)\varepsilon_\infty(M)+\\
   &-(3(\varepsilon_2(N)+2\varepsilon_3(N)-4)n-2\varepsilon_3(N)+2(1+(-1)^n)(\xi(-3)\varepsilon_3(M)-9\delta(1,M)))d(N)(p-1)+\\
   &+4(12-3\varepsilon_2(N)-4\varepsilon_3(N))d(N)+144d(p^n)\varepsilon_\infty(M)^2+12(n(p-1)+2)(3\varepsilon_2(N)+4\varepsilon_3(N)-12)\varepsilon_\infty(N)+\\
   &+2(n(p-1)+2)(9\varepsilon_2(N)(\varepsilon_2(M)-4)+16\varepsilon_3(N)(\varepsilon_3(M)-3)+12\varepsilon_2(N)\varepsilon_3(N)),
\end{align*}
here $\delta$ in the second line is the Kronecker delta, in fact the $-9$ occurs only when we need to change the model using the blow downs. We proved that \Cref{eq:opicc} holds by using the software Mathematica (see \cite{codes} for the details). It follows that
\[
\frac{g\langle V_0, V_\infty\rangle}{g-1}-\frac{\langle V_0, V_0\rangle+\langle V_\infty, V_\infty\rangle}{2g-2}=\sum_{p\in\mathcal P}\left(ng\log p+\frac{f(p,n,M)\log p}{12d(N)(p-1)}\right)=g\log N+\sum_{p\in\mathcal P}\frac{f(p,n,M)\log p}{12d(N)(p-1)}.
\]
To conclude, we only need to prove that
\[
\sum_{p\in\mathcal P}\frac{f(p,n,M)\log p}{12d(N)(p-1)}=o(g\log N), \quad \text{for }N\to+\infty.
\]
Since $g=\frac{d(N)}{12}+o(N)$, for $N\to+\infty$ (see \Cref{eq:ogenus}), we have
\[
\frac{1}{g\log N}\sum_{p\in\mathcal P}\frac{f(p,n,M)\log p}{12d(N)(p-1)}\sim \sum_{p\in\mathcal P}\frac{f(p,n,M)\log p}{d(N)^2(p-1)\log N}.
\]
Finally, it is not hard to check that for each summand $s_f(p,n,M)$ of $f(p,n,M)$, we have
\[
\lim_{N\to+\infty}\frac{s_f(p,n,M)\log p}{d(N)^2(p-1)\log N}=0.
\]
\end{proof}

\begin{remark}
When $\varepsilon_2(N)=\varepsilon_3(N)=0$, the summand (c) of \Cref{eq:mainformula} is zero (see \Cref{eq:h_m}); hence the expressions of
\[
\frac{g\langle V_0, V_\infty\rangle}{g-1}-\frac{\langle V_0, V_0\rangle+\langle V_\infty, V_\infty\rangle}{2g-2},
\]
computed in the proof of \Cref{thm:main_res} above, are the exact value of the finite part of $\langle\overline{\omega}, \overline{\omega}\rangle$. 
\end{remark}
\begin{question}
For the curves $X_0(p^n)$ with $p>3$ a prime and $n$ even, when $p^n\to +\infty$, the results show that the asymptotic of $\langle\overline{\omega}, \overline{\omega}\rangle$ does not vary changing the model from the Edixhoven's one to the minimal regular model. Therefore we wonder whether some kind of independence of the asymptotic of $\langle\overline{\omega}, \overline{\omega}\rangle$ from the model holds more in general.
\end{question}

\begin{appendices}
\section{Drawings of the special fibres}\label{sec:app}
In this appendix we include the drawings of the special fiber over $p$ of  the Edixhoven's model of $X_0(p^nM)$ described in \Cref{sec:Edi}. We denote by $P_1,\ldots,P_k$ the $k$ supersingular points described in \Cref{eq:k}.

\begin{center}
\begin{tikzpicture}[use Hobby shortcut]

\coordinate[label=right:\,\,$P_k$] (k1) at (0,0);
\coordinate (k2) at (0,2);
\coordinate[label=right:\,\,$P_1$] (k3) at (0,4);
\node[point] at (k1){};
\node[point] at (k3) {};


\coordinate[label=left:$C'_0$] (C01) at (-4,6);
\coordinate (C03) at (1,2.5);
\coordinate[label=right:$C'_0$] (C06) at (1,2);
\coordinate (C04) at (1,1.5);
\coordinate (C05) at (-4,-0.5);
\coordinate[label=right:$C'_0$] (C02) at (-6,-7);

\draw[thin,black,out angle=90,in angle=-90] (C02)..(C05)..(k1)..(C04);
\draw[dotted,black,out angle=90,in angle=-90] (C04)..(C03);
\draw[thin,black,out angle=90,in angle=-20](C03)..(k3)..(C01);


\coordinate[label=left:$C'_a$] (Ca1) at (-2,6);
\coordinate (Ca3) at (2.5,2.5);
\coordinate[label=right:$C'_a$] (Ca11) at (2.6,2);
\coordinate (Ca4) at (2.5,1.5);
\coordinate (Ca5) at (-1,-1);
\coordinate (Ca10) at (-4,-2.5);
\coordinate (Ca6) at (-4,-3);
\coordinate (Ca7) at (-4,-4.3);
\coordinate (Ca8) at (-4,-5);
\coordinate (Ca9) at (-4,-6.3);
\coordinate[label=right:$C'_a$] (Ca2) at (-4,-7);
\coordinate [label=right:$0<a<\frac{n}{2}$] (Calab) at (-3.5,6.5);

\coordinate (Ca6A) at (-4,-3.45);
\coordinate (Ca6B) at (-4,-3.85);

\coordinate (Ca8A) at (-4,-5.45);
\coordinate (Ca8B) at (-4,-5.85);

\draw[thin,black,out angle=90,in angle=-90] (Ca2)..(Ca9)..(Ca8B);
\draw[dotted,black,out angle=90,in angle=-90] (Ca8B)..(Ca8A);
\draw[thin,black,out angle=90,in angle=-90](Ca8A)..(Ca7)..(Ca6B);
\draw[dotted,black,out angle=90,in angle=-90] (Ca6B)..(Ca6A);
\draw[thin,black,out angle=90,in angle=-90] (Ca6A)..(Ca10)..(Ca5)..(k1)..(Ca4);
\draw[dotted,black,out angle=90,in angle=-90] (Ca4)..(Ca3);
\draw[thin,black,out angle=90,in angle=-20](Ca3)..(k3)..(Ca1);


\coordinate[label=right:$C'_{\frac{n}{2}}$] (Ca21) at (0,6);
\coordinate (Ca23) at (0,2.5);
\coordinate[label=right:$C'_{\frac{n}{2}}$] (Ca24) at (-0.05,2);
\coordinate[label=right:$C'_{\frac{n}{2}}$] (Ca22) at (0,-7);
\coordinate[label=right:(if $n$ is even)] (Ca2lab) at (-1,6.5);

\coordinate (CaExtra1) at (0, 2.5);
\coordinate (CaExtra2) at (0, 1.5);

\coordinate (CaExtra3) at (0,-3.45);
\coordinate (CaExtra4) at (0,-3.85);

\coordinate (CaExtra5) at (0,-5.45);
\coordinate (CaExtra6) at (0,-5.85);

\draw[thick,gray] plot[smooth] coordinates{(Ca21) (CaExtra1)};
\draw[dotted, gray] plot[smooth] coordinates{(CaExtra1) (CaExtra2)};
\draw[thick,gray] plot[smooth] coordinates{(CaExtra2) (CaExtra3)};
\draw[dotted,gray] plot[smooth] coordinates{(CaExtra3) (CaExtra4)};
\draw[thick,gray] plot[smooth] coordinates{(CaExtra4) (CaExtra5)};
\draw[dotted,gray] plot[smooth] coordinates{(CaExtra5) (CaExtra6)};
\draw[thick,gray] plot[smooth] coordinates{(CaExtra6)  (Ca22)};


\coordinate[label=right:$C'_b$] (Cb1) at (2,6);
\coordinate (Cb3) at (-2.5,2.5);
\coordinate[label=left:$C'_b$] (Cb11) at (-2.6,2);
\coordinate (Cb4) at (-2.5,1.5);
\coordinate (Cb5) at (1,-1);
\coordinate (Cb10) at (4,-2.5);
\coordinate (Cb6) at (4,-3);
\coordinate (Cb7) at (4,-4.3);
\coordinate (Cb8) at (4,-5);
\coordinate (Cb9) at (4,-6.3);
\coordinate[label=right:$C'_b$] (Cb2) at (4,-7);
\coordinate[label=right:$\frac{n}{2}<b<n$] (Cblab) at (1.5,6.5);

\coordinate (Cb6A) at (4,-3.45);
\coordinate (Cb6B) at (4,-3.85);

\coordinate (Cb8A) at (4,-5.45);
\coordinate (Cb8B) at (4,-5.85);

\draw[thin,black,out angle=90,in angle=-90] (Cb2)..(Cb9)..(Cb8B);
\draw[dotted,black,out angle=90,in angle=-90] (Cb8B)..(Cb8A);
\draw[thin,black,out angle=90,in angle=-90](Cb8A)..(Cb7)..(Cb6B);
\draw[dotted,black,out angle=90,in angle=-90] (Cb6B)..(Cb6A);
\draw[thin,black,out angle=90,in angle=-90] (Cb6A)..(Cb10)..(Cb5)..(k1)..(Cb4);
\draw[dotted,black,out angle=90,in angle=-90] (Cb4)..(Cb3);
\draw[thin,black,out angle=90,in angle=200](Cb3)..(k3)..(Cb1);


\coordinate[label=right:$C'_n$] (Cn1) at (4,6);
\coordinate (Cn3) at (-1,2.5);
\coordinate[label=left:$C'_n$] (Cn6) at (-1,2);
\coordinate (Cn4) at (-1,1.5);
\coordinate (Cn5) at (6,-1);
\coordinate[label=right:$C'_n$] (Cn2) at (7,-7);

\draw[thin,black,out angle=90,in angle=-90] (Cn2)..(Cn5)..(k1)..(Cn4);
\draw[dotted,black,out angle=90,in angle=-90] (Cn4)..(Cn3);
\draw[thin,black,out angle=90,in angle=200](Cn3)..(k3)..(Cn1);


\coordinate (Eai0) at (-5,-3);
\coordinate (Eai1) at (-4,-3);
\node[point] at (Eai1) {};
\coordinate[label=right:$E'_{a,1}$] (Eai2) at (-3,-3);
\draw[thin,black,out angle=0,in angle=180] (Eai0)..(Eai1)..(Eai2);

\coordinate[label=right:$\vdots$] (Q3) at (-3,-3.5);

\coordinate (Eaf0) at (-5,-4.3);
\coordinate (Eaf1) at (-4,-4.3);
\node[point] at (Eaf1) {};
\coordinate[label=right:$E'_{a,{\varepsilon_2(M)}}$] (Eaf2) at (-3,-4.3);
\draw[thin,black,out angle=0,in angle=180] (Eaf0)..(Eaf1)..(Eaf2);

\coordinate (Ea2i0) at (-1,-3);
\coordinate (Ea2i1) at (0,-3);
\coordinate[label=right:$E'_{\frac{n}{2},1}$](Ea2i2) at (1,-3);
\draw[thick,gray,out angle=0,in angle=180] (Ea2i0)..(Ea2i1) ..(Ea2i2);
\node[pointG] at (Ea2i1) {};

\coordinate[label=right:$\vdots$] (Q3) at (1,-3.5);

\coordinate (Ea2f0) at (-1,-4.3);
\coordinate (Ea2f1) at (0,-4.3);
\coordinate[label=right:$E'_{\frac{n}{2},{\varepsilon_2(M)}}$] (Ea2f2) at (1,-4.3);
\draw[thick,gray,out angle=0,in angle=180] (Ea2f0)..(Ea2f1)..(Ea2f2);
\node[pointG] at (Ea2f1) {};

\coordinate (Ebi0) at (3,-3);
\coordinate (Ebi1) at (4,-3);
\coordinate[label=right:$E'_{b,1}$] (Ebi2) at (5,-3);
\draw[thin,black,out angle=0,in angle=180] (Ebi0)..(Ebi1)..(Ebi2);
\node[point] at (Ebi1) {};

\coordinate[label=right:$\vdots$] (Q3) at (5,-3.5);

\coordinate (Ebf0) at (3,-4.3);
\coordinate (Ebf1) at (4,-4.3);
\node[point] at (Ebf1) {};
\coordinate[label=right:$E'_{b,{\varepsilon_2(M)}}$] (Ebf2) at (5,-4.3);
\draw[thin,black,out angle=0,in angle=180] (Ebf0)..(Ebf1)..(Ebf2);



\coordinate (Fai0) at (-5,-5);
\coordinate (Fai1) at (-4,-5);
\node[point] at (Fai1) {};
\coordinate[label=right:$F'_{a,1}$] (Fai2) at (-3,-5);
\draw[thin,black,out angle=0,in angle=180] (Fai0)..(Fai1)..(Fai2);

\coordinate[label=right:$\vdots$](Q3) at (-3,-5.5);

\coordinate (Faf0) at (-5,-6.3);
\coordinate (Faf1) at (-4,-6.3);
\node[point] at (Faf1) {};
\coordinate[label=right:$F'_{a,{\varepsilon_3(M)}}$] (Faf2) at (-3,-6.3);
\draw[thin,black,out angle=0,in angle=180] (Faf0)..(Faf1)..(Faf2);

\coordinate (Fa2i0) at (-1,-5);
\coordinate (Fa2i1) at (0,-5);
\coordinate[label=right:$F'_{\frac{n}{2},1}$] (Fa2i2) at (1,-5);
\draw[thin,gray,out angle=0,in angle=180] (Fa2i0)..(Fa2i1)..(Fa2i2);
\node[pointG] at (Fa2i1) {};

\coordinate[label=right:$\vdots$] (Q3) at (1,-5.5);

\coordinate (Fa2f0) at (-1,-6.3);
\coordinate (Fa2f1) at (0,-6.3);
\coordinate[label=right:$F'_{\frac{n}{2},{\varepsilon_3(M)}}$] (Fa2f2) at (1,-6.3);
\draw[thin,gray,out angle=0,in angle=180] (Fa2f0)..(Fa2f1)..(Fa2f2);
\node[pointG] at (Fa2f1) {};

\coordinate (Fbi0) at (3,-5);
\coordinate (Fbi1) at (4,-5);
\node[point] at (Fbi1) {};
\coordinate[label=right:$F'_{b,1}$] (Fbi2) at (5,-5);
\draw[thin,black,out angle=0,in angle=180] (Fbi0)..(Fbi1)..(Fbi2);

\coordinate[label=right:$\vdots$](Q3) at (5,-5.5);

\coordinate (Fbf0) at (3,-6.3);
\coordinate (Fbf1) at (4,-6.3);
\node[point] at (Fbf1) {};
\coordinate[label=right:$F'_{b,{\varepsilon_3(M)}}$] (Fbf2) at (5,-6.3);
\draw[thin,black,out angle=0,in angle=180] (Fbf0)..(Fbf1)..(Fbf2);


\end{tikzpicture}
\captionof{figure}{Case $p \equiv 1 \mod 12$.}
\label{fig:fig1}
\end{center}

\begin{center}

\begin{tikzpicture}[use Hobby shortcut]

\coordinate[label=right:\,\,\,\,$P_k$] (k1) at (0,0);
\coordinate (k2) at (0,2);
\coordinate[label=right:\,\,$P_1$] (k3) at (0,4);
\node[point] at (k1) {};
\node[point] at (k3) {};

\coordinate[label=left:$C'_0$] (C01) at (-4,6);
\coordinate (C03) at (1,2.5);
\coordinate[label=right:$C'_0$] (C05) at (1,2);
\coordinate (C04) at (1,1.5);
\coordinate[label=left:$C'_0$] (C06) at (-6,-1.5);
\coordinate (C07) at (-6,-3);
\coordinate (C09) at (-6,-3.65);
\coordinate (C08) at (-6,-4.3);
\coordinate[label=left:$C'_0$] (C02) at (-6,-8);

\coordinate (C09A) at (-6,-3.85);
\coordinate (C09B) at (-6,-3.45);

\draw[thin,black,out angle=90,in angle=-90] (C02)..(C08)..(C09A);
\draw[dotted,black,out angle=90,in angle=-90]  (C09A)..(C09B);
\draw[thin,black,out angle=90,in angle=-90] (C09B)..(C07)..(C06)..(k1)..(C04);
\draw [dotted, black, out angle=90,in angle=-90] (C04)..(C03);
\draw [thin, black, out angle=90,in angle=-20](C03)..(k3)..(C01);


\coordinate[label=left:$C'_a$] (Ca1) at (-2,6);
\coordinate (Ca3) at (2.5,2.5);
\coordinate[label=right:$C'_a$] (Ca14) at (2.6,2);
\coordinate (Ca4) at (2.5,1.5);
\coordinate (Ca5) at (-1,-1);
\coordinate[label=left:$C'_a$] (Ca13) at (-2,-1.5);
\coordinate (Ca11) at (-5.5,-2.7);
\coordinate (Ca10) at (-6,-3);
\coordinate[label=right:$C'_a$] (Ca6) at (-5.7,-3.65);
\coordinate (Ca7) at (-6,-4.3);
\coordinate (Ca12) at (-5.5,-4.6);
\coordinate (Ca8) at (-4,-6);
\coordinate (Ca9) at (-4,-7.3);
\coordinate[label=right:$C'_a$] (Ca2) at (-4,-8);

\coordinate (Ca6A) at (-5.7,-3.85);
\coordinate (Ca6B) at (-5.7,-3.45);

\coordinate (Caa) at (-4,-6.9);
\coordinate (Caa1) at (-4,-6.5);
\coordinate [label=right:$0<a<\frac{n}{2}$] (Calab) at (-3.5,6.5);

\draw[thin,black,out angle=90,in angle=-90] (Ca2)..(Ca9)..(Caa);
\draw[dotted,black,out angle=90,in angle=-90] (Caa)..(Caa1);
\draw[thin,black,out angle=90,in angle=-90] (Caa1)..(Ca8)..(Ca12)..(Ca7)..(Ca6A);
\draw[dotted,black,out angle=90,in angle=-90] (Ca6A)..(Ca6B);
\draw[thin,black,out angle=90,in angle=-90](Ca6B)..(Ca10)..(Ca11)..(Ca5)..(k1)..(Ca4);
\draw[dotted,black,out angle=90,in angle=-90](Ca4)..(Ca3);
\draw[thin,black,out angle=90,in angle=-20](Ca3)..(k3)..(Ca1);



\coordinate[label=right:$C'_b$] (Cb1) at (2,6);
\coordinate (Cb3) at (-2.5,2.5);
\coordinate[label=left:$C'_b$] (Cb11) at (-2.6,2);
\coordinate (Cb4) at (-2.5,1.5);
\coordinate (Cb5) at (1,-1);
\coordinate[label=right:$C'_b$] (Cb13) at (1.8,-1.5);
\coordinate (Cb10) at (6.5,-2.7);
\coordinate (Cb6) at (7,-3);
\coordinate[label=left:$C'_b$] (Cb12) at (6.7,-3.65);
\coordinate (Cb7) at (7,-4.3);
\coordinate (Cb14) at (6.5,-4.7);
\coordinate (Cb8) at (4,-6);
\coordinate (Cb9) at (4,-7.3);
\coordinate[label=right:$C'_b$] (Cb2) at (4,-8);
\coordinate[label=right:$\frac{n}{2}<b<n$] (Cblab) at (1.5,6.5);

\coordinate (Cb12A) at (6.7,-3.85);
\coordinate (Cb12B) at (6.7,-3.45);

\coordinate (Cbb) at (4,-6.9);
\coordinate (Cbb1) at (4,-6.5);

\draw[thin,black,out angle=90,in angle=-90] (Cb2)..(Cb9)..(Cbb);
\draw[dotted,black,out angle=90,in angle=-90] (Cbb)..(Cbb1);
\draw[thin,black,out angle=90,in angle=-90] (Cbb1)..(Cb8)..(Cb14)..(Cb7)..(Cb12A);
\draw[dotted,black,out angle=90,in angle=-90] (Cb12A)..(Cb12B);
\draw[thin,black,out angle=90,in angle=-90] (Cb12B)..(Cb6)..(Cb10)..(Cb5)..(k1)..(Cb4);
\draw[dotted,black,out angle=90,in angle=-90] (Cb4)..(Cb3);
\draw[thin,black,out angle=90,in angle=200] (Cb3)..(k3)..(Cb1);

\coordinate[label=right:$C'_n$] (Cn1) at (4,6);
\coordinate (Cn3) at (-1,2.5);
\coordinate[label=left:$C'_n$] (Cn5) at (-1,2);
\coordinate (Cn4) at (-1,1.5);
\coordinate[label=right:$C'_n$] (Cn6) at (7,-1.5);
\coordinate (Cn7) at (7,-3);
\coordinate (Cn9) at (7,-3.65);
\coordinate (Cn8) at (7,-4.3);
\coordinate[label=right:$C'_n$] (Cn2) at (7,-8);

\coordinate (Cn9A) at (7,-3.85);
\coordinate (Cn9B) at (7,-3.45);

\draw[thin,black,out angle=90,in angle=-90] (Cn2)..(Cn8)..(Cn9A);
\draw[dotted,black,out angle=90,in angle=-90]  (Cn9A)..(Cn9B);
\draw[thin,black,out angle=90,in angle=-90] (Cn9B)..(Cn7)..(Cn6)..(k1)..(Cn4);
\draw [dotted, black, out angle=90,in angle=-90] (Cn4)..(Cn3);
\draw [thin, black, out angle=90,in angle=200](Cn3)..(k3)..(Cn1);



\coordinate (F01a) at (-7,-2.83);
\coordinate (F01b) at (2,-4.4);
\draw[thin,black, name path=F01] plot[smooth] coordinates{(F01a) (F01b)};

\coordinate (Fn1a) at (8,-2.83);
\coordinate (Fn1b) at (-1,-4.4);
\draw[thin,black,name path=Fn1] plot[smooth] coordinates{(Fn1a) (Fn1b)};

\path [name intersections={of=F01 and Fn1,by=i1}];

\node[point] at (i1) {};

\coordinate (F0ea) at (-7,-4.13);
\coordinate (F0eb) at (2,-5.7);
\draw[thin,black, name path=F0e] plot[smooth] coordinates{(F0ea) (F0eb)};

\coordinate (Fnea) at (8,-4.13);
\coordinate (Fneb) at (-1,-5.7);
\draw[thin,black, name path=Fne] plot[smooth] coordinates{(Fnea) (Fneb)};

\path [name intersections={of=F0e and Fne,by=i2}];

\node[point] at (i2) {};

\coordinate[label=right:$F'_{0,1}$] (Fi4) at (-8,-3);
\coordinate (Fi1) at (-6,-3);
\node[point] at (Fi1) {};
\coordinate (Fi3) at (7,-3);
\node[point] at (Fi3) {};
\coordinate[label=right:$F'_{n,1}$] (Fi4) at (8,-3);

\coordinate[label=right:$\vdots$] (Q3) at (8,-3.5);

\coordinate[label=right:$\vdots$] (Q3) at (-8,-3.5);

\coordinate[label=right:$F'_{0,\varepsilon_3(M)}$] (Fi4) at (-8,-4.4);
\coordinate (Ff1) at (-6,-4.3);
\node[point] at (Ff1) {};
\coordinate (Ff3) at (7,-4.3);
\node[point] at (Ff3) {};
\coordinate[label=right:$F'_{n,{\varepsilon_3(M)}}$] (Ff4) at (8,-4.3);



\coordinate (Eai0) at (-5,-6);
\coordinate (Eai1) at (-4,-6);
\node[point] at (Eai1) {};
\coordinate[label=right:$E'_{a,1}$] (Eai2) at (-3,-6);
\draw[thin,black,out angle=0,in angle=180] (Eai0)..(Eai1)..(Eai2);

\coordinate[label=right:$\vdots$](Q3) at (-3,-6.5);

\coordinate (Eaf0) at (-5,-7.3);
\coordinate (Eaf1) at (-4,-7.3);
\node[point] at (Eaf1) {};
\coordinate[label=right:$E'_{a,{\varepsilon_2(M)}}$] (Eaf2) at (-3,-7.3);
\draw[thin,black,out angle=0,in angle=180] (Eaf0)..(Eaf1)..(Eaf2);

\coordinate (Ebi0) at (3,-6);
\coordinate (Ebi1) at (4,-6);
\node[point] at (Ebi1) {};
\coordinate[label=right:$E'_{b,1}$] (Ebi2) at (5,-6);
\draw[thin,black,out angle=0,in angle=180] (Ebi0)..(Ebi1)..(Ebi2);

\coordinate[label=right:$\vdots$](Q3) at (5,-6.5);

\coordinate (Ebf0) at (3,-7.3);
\coordinate (Ebf1) at (4,-7.3);
\node[point] at (Ebf1) {};
\coordinate[label=right:$E'_{b,{\varepsilon_2(M)}}$] (Ebf2) at (5,-7.3);
\draw[thin,black,out angle=0,in angle=180] (Ebf0)..(Ebf1)..(Ebf2);


\end{tikzpicture}


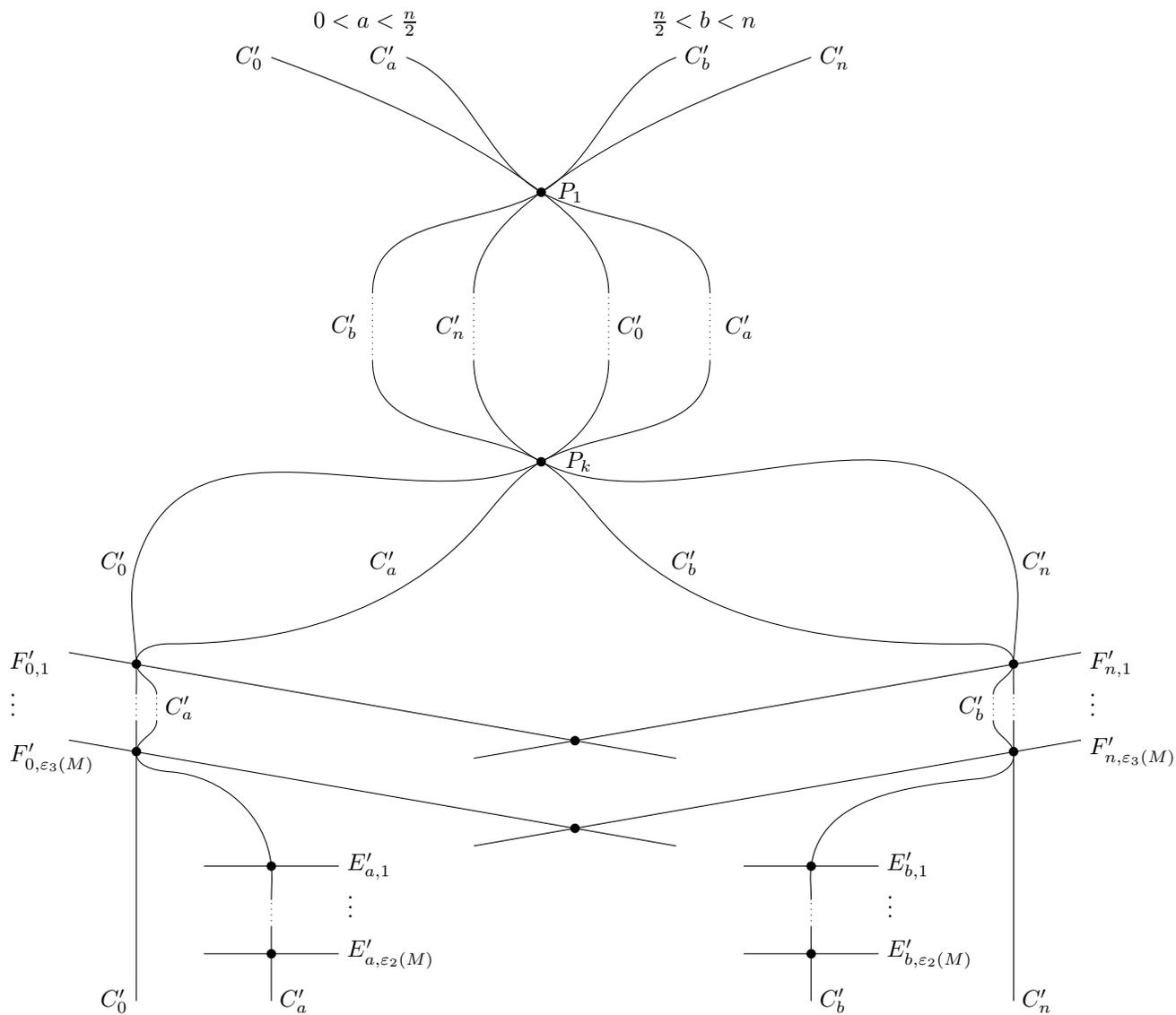
\captionof{figure}{Case $p \equiv 5 \mod 12$, $n$ odd.}
\label{fig:fig3}
\end{center}

\begin{center}
\begin{tikzpicture}[use Hobby shortcut]
\coordinate[label=right:\,\,\,\,$P_k$] (k1) at (0,0);
\coordinate (k2) at (0,2);
\coordinate[label=right:\,\,$P_1$] (k3) at (0,4);
\node[point] at (k1) {};
\node[point] at (k3) {};

\coordinate[label=left:$C'_0$] (C01) at (-4,6);
\coordinate (C03) at (1,2.5);
\coordinate[label=right:$C'_0$] (C05) at (1,2);
\coordinate (C04) at (1,1.5);
\coordinate[label=left:$C'_0$] (C06) at (-6,-1.5);
\coordinate (C07) at (-6,-3);
\coordinate (C09) at (-6,-3.65);
\coordinate (C08) at (-6,-4.3);
\coordinate[label=left:$C'_0$] (C02) at (-6,-8);

\coordinate (C09A) at (-6,-3.85);
\coordinate (C09B) at (-6,-3.45);

\draw[thin,black,out angle=90,in angle=-90] (C02)..(C08)..(C09A);
\draw[dotted,black,out angle=90,in angle=-90]  (C09A)..(C09B);
\draw[thin,black,out angle=90,in angle=-90] (C09B)..(C07)..(C06)..(k1)..(C04);
\draw [dotted, black, out angle=90,in angle=-90] (C04)..(C03);
\draw [thin, black, out angle=90,in angle=-20](C03)..(k3)..(C01);


\coordinate[label=left:$C'_a$] (Ca1) at (-2,6);
\coordinate (Ca3) at (2.5,2.5);
\coordinate[label=right:$C'_a$] (Ca14) at (2.6,2);
\coordinate (Ca4) at (2.5,1.5);
\coordinate (Ca5) at (-1,-1);
\coordinate[label=left:$C'_a$] (Ca13) at (-2,-1.5);
\coordinate (Ca11) at (-5.5,-2.7);
\coordinate (Ca10) at (-6,-3);
\coordinate[label=right:$C'_a$] (Ca6) at (-5.7,-3.65);
\coordinate (Ca7) at (-6,-4.3);
\coordinate (Ca12) at (-5.5,-4.5);
\coordinate (Ca8) at (-4,-6);
\coordinate (Ca9) at (-4,-7.3);
\coordinate[label=right:$C'_a$] (Ca2) at (-4,-8);

\coordinate (Ca6A) at (-5.7,-3.85);
\coordinate (Ca6B) at (-5.7,-3.45);

\coordinate (Caa) at (-4,-6.9);
\coordinate (Caa1) at (-4,-6.5);

\coordinate [label=right:$0<a<\frac{n}{2}$] (Calab) at (-3.5,6.5);

\draw[thin,black,out angle=90,in angle=-90] (Ca2)..(Ca9)..(Caa);
\draw[dotted,black,out angle=90,in angle=-90] (Caa)..(Caa1);
\draw[thin,black,out angle=90,in angle=-90] (Caa1)..(Ca8)..(Ca12)..(Ca7)..(Ca6A);
\draw[dotted,black,out angle=90,in angle=-90] (Ca6A)..(Ca6B);
\draw[thin,black,out angle=90,in angle=-90](Ca6B)..(Ca10)..(Ca11)..(Ca5)..(k1)..(Ca4);
\draw[dotted,black,out angle=90,in angle=-90](Ca4)..(Ca3);
\draw[thin,black,out angle=90,in angle=-20](Ca3)..(k3)..(Ca1);


\coordinate[label=right:$C'_{\frac{n}{2}}$] (Ca21) at (0,6);
\coordinate (Ca23) at (0,2.5);
\coordinate[label=right:$C'_{\frac{n}{2}}$] (Ca24) at (-0.05,2);
\coordinate[label=right:$C'_{\frac{n}{2}}$] (Ca25) at (0,-3.65);
\coordinate[label=right:$C'_{\frac{n}{2}}$] (Ca22) at (0,-8);

\coordinate (CaExtra1) at (0, 1.5);
\coordinate (CaExtra2) at (0, 2.5);

\coordinate (Ca25A) at (0,-3.85);
\coordinate (Ca25B) at (0,-3.45);

\coordinate (Caa2) at (0,-6.9);
\coordinate (Caa3) at (0,-6.5);

\draw[thin,black] plot[smooth] coordinates{(Ca22) (Caa2)};
\draw[dotted, black] plot[smooth] coordinates{(Caa2) (Caa3)};
\draw[thin,black] plot[smooth] coordinates{(Caa3) (Ca25A)};
\draw[dotted,black] plot[smooth] coordinates{(Ca25A) (Ca25B)};
\draw[thin,black] plot[smooth] coordinates{(Ca25B) (CaExtra1)};
\draw[dotted, black] plot[smooth] coordinates{(CaExtra1) (CaExtra2)};
\draw[thin,black] plot[smooth] coordinates{(CaExtra2) (Ca21)};


\coordinate[label=right:$C'_b$] (Cb1) at (2,6);
\coordinate (Cb3) at (-2.5,2.5);
\coordinate[label=left:$C'_b$] (Cb11) at (-2.6,2);
\coordinate (Cb4) at (-2.5,1.5);
\coordinate (Cb5) at (1,-1);
\coordinate[label=right:$C'_b$] (Cb13) at (1.8,-1.5);
\coordinate (Cb10) at (6.5,-2.7);
\coordinate (Cb6) at (7,-3);
\coordinate[label=left:$C'_b$] (Cb12) at (6.7,-3.65);
\coordinate (Cb7) at (7,-4.3);
\coordinate (Cb14) at (6.5,-4.5);
\coordinate (Cb8) at (4,-6);
\coordinate (Cb9) at (4,-7.3);
\coordinate[label=right:$C'_b$] (Cb2) at (4,-8);
\coordinate[label=right:$\frac{n}{2}<b<n$] (Cblab) at (1.5,6.5);

\coordinate (Cb12A) at (6.7,-3.85);
\coordinate (Cb12B) at (6.7,-3.45);

\coordinate (Cbb) at (4,-6.9);
\coordinate (Cbb1) at (4,-6.5);

\draw[thin,black,out angle=90,in angle=-90] (Cb2)..(Cb9)..(Cbb);
\draw[dotted,black,out angle=90,in angle=-90] (Cbb)..(Cbb1);
\draw[thin,black,out angle=90,in angle=-90] (Cbb1)..(Cb8)..(Cb14)..(Cb7)..(Cb12A);
\draw[dotted,black,out angle=90,in angle=-90] (Cb12A)..(Cb12B);
\draw[thin,black,out angle=90,in angle=-90] (Cb12B)..(Cb6)..(Cb10)..(Cb5)..(k1)..(Cb4);
\draw[dotted,black,out angle=90,in angle=-90] (Cb4)..(Cb3);
\draw[thin,black,out angle=90,in angle=200] (Cb3)..(k3)..(Cb1);

\coordinate[label=right:$C'_n$] (Cn1) at (4,6);
\coordinate (Cn3) at (-1,2.5);
\coordinate[label=left:$C'_n$] (Cn5) at (-1,2);
\coordinate (Cn4) at (-1,1.5);
\coordinate[label=right:$C'_n$] (Cn6) at (7,-1.5);
\coordinate (Cn7) at (7,-3);
\coordinate(Cn9) at (7,-3.65);
\coordinate (Cn8) at (7,-4.3);
\coordinate[label=right:$C'_n$] (Cn2) at (7,-8);

\coordinate (Cn9A) at (7,-3.85);
\coordinate (Cn9B) at (7,-3.45);

\draw[thin,black,out angle=90,in angle=-90] (Cn2)..(Cn8)..(Cn9A);
\draw[dotted,black,out angle=90,in angle=-90]  (Cn9A)..(Cn9B);
\draw[thin,black,out angle=90,in angle=-90] (Cn9B)..(Cn7)..(Cn6)..(k1)..(Cn4);
\draw [dotted, black, out angle=90,in angle=-90] (Cn4)..(Cn3);
\draw [thin, black, out angle=90,in angle=200](Cn3)..(k3)..(Cn1);



\coordinate (Fi0) at (-7,-3);
\coordinate (Fi1) at (-6,-3);
\node[point] at (Fi1) {};
\coordinate (Fi2) at (0,-3);
\node[point] at (Fi2) {};
\coordinate (Fi3) at (7,-3);
\node[point] at (Fi3) {};
\coordinate[label=right:$F'_{\infty,1}$] (Fi4) at (8,-3);
\draw[thin,black,out angle=0,in angle=180] (Fi0)..(Fi1)..(Fi2)..(Fi3)..(Fi4);

\coordinate[label=right:$\vdots$] (Q3) at (8,-3.5);

\coordinate (Ff0) at (-7,-4.3);
\coordinate (Ff1) at (-6,-4.3);
\node[point] at (Ff1) {};
\coordinate (Ff2) at (0,-4.3);
\node[point] at (Ff2) {};
\coordinate (Ff3) at (7,-4.3);
\node[point] at (Ff3) {};
\coordinate[label=right:$F'_{\infty,{\varepsilon_3(M)}}$] (Ff4) at (8,-4.3);
\draw[thin,black,out angle=0,in angle=180] (Ff0)..(Ff1)..(Ff2)..(Ff3)..(Ff4);



\coordinate (Eai0) at (-5,-6);
\coordinate (Eai1) at (-4,-6);
\node[point] at (Eai1) {};
\coordinate[label=right:$E'_{a,1}$] (Eai2) at (-3,-6);
\draw[thin,black,out angle=0,in angle=180] (Eai0)..(Eai1)..(Eai2);

\coordinate[label=right:$\vdots$](Q3) at (-3,-6.5);

\coordinate (Eaf0) at (-5,-7.3);
\coordinate (Eaf1) at (-4,-7.3);
\node[point] at (Eaf1) {};
\coordinate[label=right:$E'_{a,{\varepsilon_2(M)}}$] (Eaf2) at (-3,-7.3);
\draw[thin,black,out angle=0,in angle=180] (Eaf0)..(Eaf1)..(Eaf2);

\coordinate (Ea2i0) at (-1,-6);
\coordinate (Ea2i1) at (0,-6);
\node[point] at (Ea2i1) {};
\coordinate[label=right:$E'_{\frac{n}{2},1}$] (Ea2i2) at (1,-6);
\draw[thin,black,out angle=0,in angle=180] (Ea2i0)..(Ea2i1)..(Ea2i2);

\coordinate[label=right:$\vdots$] (Q3) at (1,-6.5);

\coordinate (Ea2f0) at (-1,-7.3);
\coordinate (Ea2f1) at (0,-7.3);
\node[point] at (Ea2f1) {};
\coordinate[label=right:$E'_{\frac{n}{2},{\varepsilon_2(M)}}$] (Ea2f2) at (1,-7.3);
\draw[thin,black,out angle=0,in angle=180] (Ea2f0)..(Ea2f1)..(Ea2f2);

\coordinate (Ebi0) at (3,-6);
\coordinate (Ebi1) at (4,-6);
\node[point] at (Ebi1) {};
\coordinate[label=right:$E'_{b,1}$] (Ebi2) at (5,-6);
\draw[thin,black,out angle=0,in angle=180] (Ebi0)..(Ebi1)..(Ebi2);

\coordinate[label=right:$\vdots$](Q3) at (5,-6.5);

\coordinate (Ebf0) at (3,-7.3);
\coordinate (Ebf1) at (4,-7.3);
\node[point] at (Ebf1) {};
\coordinate[label=right:$E'_{b,{\varepsilon_2(M)}}$] (Ebf2) at (5,-7.3);
\draw[thin,black,out angle=0,in angle=180] (Ebf0)..(Ebf1)..(Ebf2);

\end{tikzpicture}


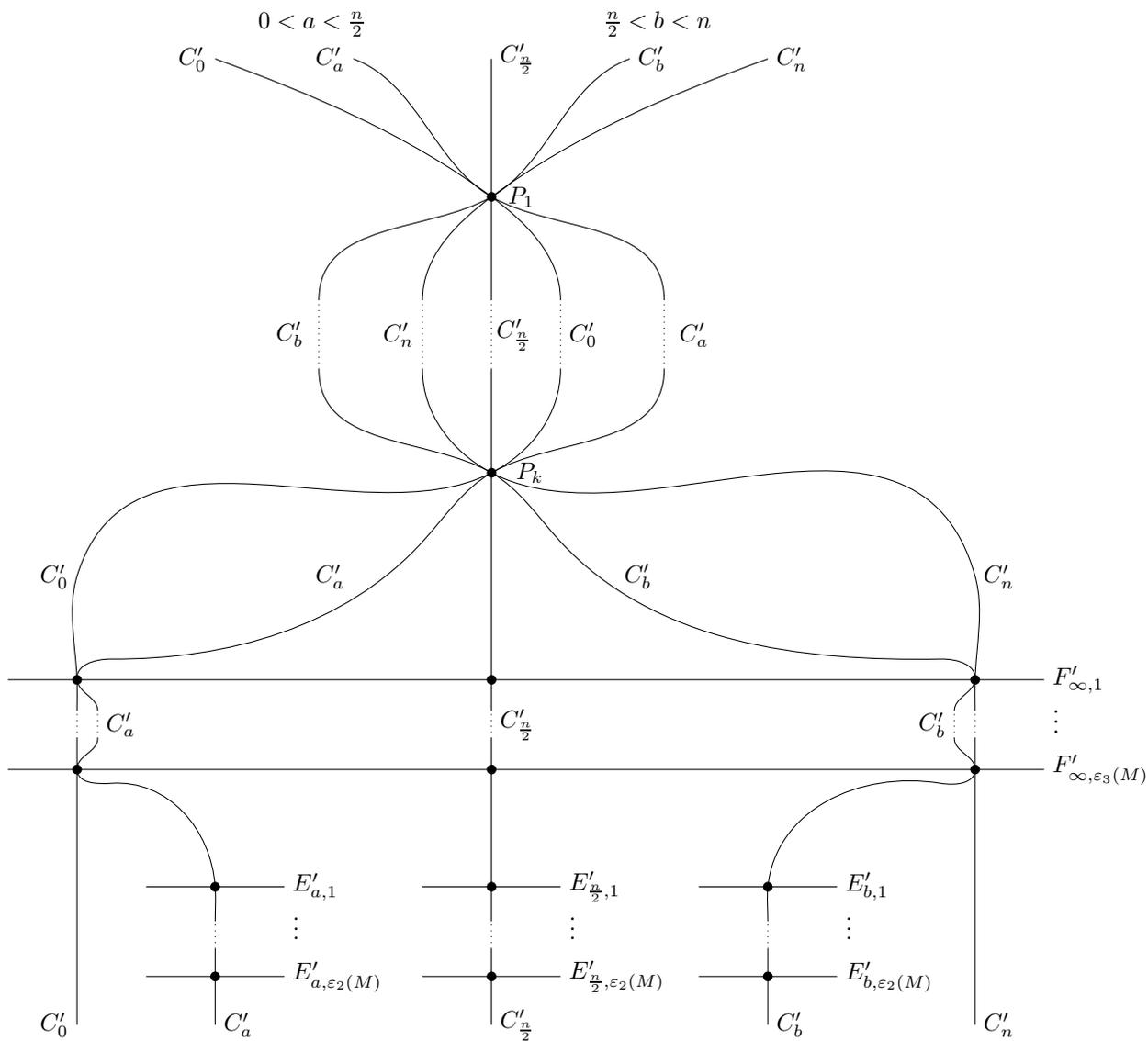
\captionof{figure}{Case $p \equiv 5 \mod 12$, $n$ even.}
\label{fig:fig2}
\end{center}

\begin{center}
\begin{tikzpicture}[use Hobby shortcut]

\coordinate[label=right:\,\,\,\,$P_k$] (k1) at (0,0);
\coordinate (k2) at (0,2);
\coordinate[label=right:\,\,$P_1$] (k3) at (0,4);
\node[point] at (k1) {};
\node[point] at (k3) {};

\coordinate[label=left:$C'_0$] (C01) at (-4,6);
\coordinate (C03) at (1,2.5);
\coordinate[label=right:$C'_0$] (C05) at (1,2);
\coordinate (C04) at (1,1.5);
\coordinate[label=left:$C'_0$] (C06) at (-6,-1.5);
\coordinate (C07) at (-6,-3);
\coordinate (C09) at (-6,-3.65);
\coordinate (C08) at (-6,-4.3);
\coordinate[label=left:$C'_0$] (C02) at (-6,-8);

\coordinate (C09A) at (-6,-3.85);
\coordinate (C09B) at (-6,-3.45);

\draw[thin,black,out angle=90,in angle=-90] (C02)..(C08)..(C09A);
\draw[dotted,black,out angle=90,in angle=-90]  (C09A)..(C09B);
\draw[thin,black,out angle=90,in angle=-90] (C09B)..(C07)..(C06)..(k1)..(C04);
\draw [dotted, black, out angle=90,in angle=-90] (C04)..(C03);
\draw [thin, black, out angle=90,in angle=-20](C03)..(k3)..(C01);


\coordinate[label=left:$C'_a$] (Ca1) at (-2,6);
\coordinate (Ca3) at (2.5,2.5);
\coordinate[label=right:$C'_a$] (Ca14) at (2.6,2);
\coordinate (Ca4) at (2.5,1.5);
\coordinate (Ca5) at (-1,-1);
\coordinate[label=left:$C'_a$] (Ca13) at (-2,-1.5);
\coordinate (Ca11) at (-5.5,-2.7);
\coordinate (Ca10) at (-6,-3);
\coordinate[label=right:$C'_a$] (Ca6) at (-5.7,-3.65);
\coordinate (Ca7) at (-6,-4.3);
\coordinate (Ca12) at (-5.5,-4.5);
\coordinate (Ca8) at (-4,-6);
\coordinate (Ca9) at (-4,-7.3);
\coordinate[label=right:$C'_a$] (Ca2) at (-4,-8);

\coordinate (Ca6A) at (-5.7,-3.85);
\coordinate (Ca6B) at (-5.7,-3.45);

\coordinate (Caa) at (-4,-6.9);
\coordinate (Caa1) at (-4,-6.5);

\coordinate [label=right:$0<a<\frac{n}{2}$] (Calab) at (-3.5,6.5);

\draw[thin,black,out angle=90,in angle=-90] (Ca2)..(Ca9)..(Caa);
\draw[dotted,black,out angle=90,in angle=-90] (Caa)..(Caa1);
\draw[thin,black,out angle=90,in angle=-90] (Caa1)..(Ca8)..(Ca12)..(Ca7)..(Ca6A);
\draw[dotted,black,out angle=90,in angle=-90] (Ca6A)..(Ca6B);
\draw[thin,black,out angle=90,in angle=-90](Ca6B)..(Ca10)..(Ca11)..(Ca5)..(k1)..(Ca4);
\draw[dotted,black,out angle=90,in angle=-90](Ca4)..(Ca3);
\draw[thin,black,out angle=90,in angle=-20](Ca3)..(k3)..(Ca1);


\coordinate[label=right:$C'_{\frac{n}{2}}$] (Ca21) at (0,6);
\coordinate (Ca23) at (0,2.5);
\coordinate[label=right:$C'_{\frac{n}{2}}$] (Ca24) at (-0.05,2);
\coordinate[label=right:$C'_{\frac{n}{2}}$] (Ca25) at (0,-3.65);
\coordinate[label=right:$C'_{\frac{n}{2}}$] (Ca22) at (0,-8);

\coordinate (CaExtra1) at (0, 1.5);
\coordinate (CaExtra2) at (0, 2.5);

\coordinate (Ca25A) at (0,-3.85);
\coordinate (Ca25B) at (0,-3.45);

\coordinate (Caa2) at (0,-6.9);
\coordinate (Caa3) at (0,-6.5);

\coordinate[label=right:(if $n$ is even)] (Ca2lab) at (-1,6.5);

\draw[thick,gray] plot[smooth] coordinates{(Ca22) (Caa2)};
\draw[dotted, gray] plot[smooth] coordinates{(Caa2) (Caa3)};
\draw[thick,gray] plot[smooth] coordinates{(Caa3) (Ca25A)};
\draw[dotted,gray] plot[smooth] coordinates{(Ca25A) (Ca25B)};
\draw[thick,gray] plot[smooth] coordinates{(Ca25B) (CaExtra1)};
\draw[dotted, gray] plot[smooth] coordinates{(CaExtra1) (CaExtra2)};
\draw[thick,gray] plot[smooth] coordinates{(CaExtra2) (Ca21)};


\coordinate[label=right:$C'_b$] (Cb1) at (2,6);
\coordinate (Cb3) at (-2.5,2.5);
\coordinate[label=left:$C'_b$] (Cb11) at (-2.6,2);
\coordinate (Cb4) at (-2.5,1.5);
\coordinate (Cb5) at (1,-1);
\coordinate[label=right:$C'_b$] (Cb13) at (1.8,-1.5);
\coordinate (Cb10) at (6.5,-2.7);
\coordinate (Cb6) at (7,-3);
\coordinate[label=left:$C'_b$] (Cb12) at (6.7,-3.65);
\coordinate (Cb7) at (7,-4.3);
\coordinate (Cb14) at (6.5,-4.5);
\coordinate (Cb8) at (4,-6);
\coordinate (Cb9) at (4,-7.3);
\coordinate[label=right:$C'_b$] (Cb2) at (4,-8);
\coordinate[label=right:$\frac{n}{2}<b<n$] (Cblab) at (1.5,6.5);

\coordinate (Cb12A) at (6.7,-3.85);
\coordinate (Cb12B) at (6.7,-3.45);

\coordinate (Cbb) at (4,-6.9);
\coordinate (Cbb1) at (4,-6.5);

\draw[thin,black,out angle=90,in angle=-90] (Cb2)..(Cb9)..(Cbb);
\draw[dotted,black,out angle=90,in angle=-90] (Cbb)..(Cbb1);
\draw[thin,black,out angle=90,in angle=-90] (Cbb1)..(Cb8)..(Cb14)..(Cb7)..(Cb12A);
\draw[dotted,black,out angle=90,in angle=-90] (Cb12A)..(Cb12B);
\draw[thin,black,out angle=90,in angle=-90] (Cb12B)..(Cb6)..(Cb10)..(Cb5)..(k1)..(Cb4);
\draw[dotted,black,out angle=90,in angle=-90] (Cb4)..(Cb3);
\draw[thin,black,out angle=90,in angle=200] (Cb3)..(k3)..(Cb1);

\coordinate[label=right:$C'_n$] (Cn1) at (4,6);
\coordinate (Cn3) at (-1,2.5);
\coordinate[label=left:$C'_n$] (Cn5) at (-1,2);
\coordinate (Cn4) at (-1,1.5);
\coordinate[label=right:$C'_n$] (Cn6) at (7,-1.5);
\coordinate (Cn7) at (7,-3);
\coordinate (Cn9) at (7,-3.65);
\coordinate (Cn8) at (7,-4.3);
\coordinate[label=right:$C'_n$] (Cn2) at (7,-8);

\coordinate (Cn9A) at (7,-3.85);
\coordinate (Cn9B) at (7,-3.45);

\draw[thin,black,out angle=90,in angle=-90] (Cn2)..(Cn8)..(Cn9A);
\draw[dotted,black,out angle=90,in angle=-90]  (Cn9A)..(Cn9B);
\draw[thin,black,out angle=90,in angle=-90] (Cn9B)..(Cn7)..(Cn6)..(k1)..(Cn4);
\draw [dotted, black, out angle=90,in angle=-90] (Cn4)..(Cn3);
\draw [thin, black, out angle=90,in angle=200](Cn3)..(k3)..(Cn1);



\coordinate (Ei0) at (-7,-3);
\coordinate (Ei1) at (-6,-3);
\node[point] at (Ei1) {};
\coordinate (Ei2) at (0,-3);
\node[pointG] at (Ei2) {};
\coordinate (Ei3) at (7,-3);
\node[point] at (Ei3) {};
\coordinate[label=right:$E'_{\infty,1}$] (Ei4) at (8,-3);
\draw[thin,black,out angle=0,in angle=180] (Ei0)..(Ei1)..(Ei2)..(Ei3)..(Ei4);

\coordinate[label=right:$\vdots$] (Q3) at (8,-3.5);

\coordinate (Ef0) at (-7,-4.3);
\coordinate (Ef1) at (-6,-4.3);
\node[point] at (Ef1) {};
\coordinate (Ef2) at (0,-4.3);

\coordinate (Ef3) at (7,-4.3);
\node[point] at (Ef3) {};
\coordinate[label=right:$E'_{\infty,{\varepsilon_2(M)}}$] (Ef4) at (8,-4.3);

\draw[thin,black,out angle=0,in angle=180] (Ef0)..(Ef4);

\node[pointG] at (Ef2) {};


\coordinate (Fai0) at (-5,-6);
\coordinate (Fai1) at (-4,-6);
\node[point] at (Fai1) {};
\coordinate[label=right:$F'_{a,1}$] (Fai2) at (-3,-6);
\draw[thin,black,out angle=0,in angle=180] (Fai0)..(Fai1)..(Fai2);

\coordinate[label=right:$\vdots$](Q3) at (-3,-6.5);

\coordinate (Faf0) at (-5,-7.3);
\coordinate (Faf1) at (-4,-7.3);
\node[point] at (Faf1) {};
\coordinate[label=right:$F'_{a,{\varepsilon_3(M)}}$] (Faf2) at (-3,-7.3);
\draw[thin,black,out angle=0,in angle=180] (Faf0)..(Faf1)..(Faf2);

\coordinate (Fa2i0) at (-1,-6);
\coordinate (Fa2i1) at (0,-6);
\node[pointG] at (Fa2i1) {};
\coordinate[label=right:$F'_{\frac{n}{2},1}$] (Fa2i2) at (1,-6);
\draw[thick,gray,out angle=0,in angle=180] (Fa2i0)..(Fa2i1)..(Fa2i2);

\coordinate[label=right:$\vdots$] (Q3) at (1,-6.5);

\coordinate (Fa2f0) at (-1,-7.3);
\coordinate (Fa2f1) at (0,-7.3);
\node[pointG] at (Fa2f1) {};
\coordinate[label=right:$F'_{\frac{n}{2},{\varepsilon_3(M)}}$] (Fa2f2) at (1,-7.3);
\draw[thick,gray,out angle=0,in angle=180] (Fa2f0)..(Fa2f1)..(Fa2f2);

\coordinate (Fbi0) at (3,-6);
\coordinate (Fbi1) at (4,-6);
\node[point] at (Fbi1) {};
\coordinate[label=right:$F'_{b,1}$] (Fbi2) at (5,-6);
\draw[thin,black,out angle=0,in angle=180] (Fbi0)..(Fbi1)..(Fbi2);

\coordinate[label=right:$\vdots$](Q3) at (5,-6.5);

\coordinate (Fbf0) at (3,-7.3);
\coordinate (Fbf1) at (4,-7.3);
\node[point] at (Fbf1) {};
\coordinate[label=right:$F'_{b,{\varepsilon_3(M)}}$] (Fbf2) at (5,-7.3);
\draw[thin,black,out angle=0,in angle=180] (Fbf0)..(Fbf1)..(Fbf2);


\end{tikzpicture}
\captionof{figure}{Case $p \equiv 7 \mod 12$.}
\label{fig:fig4}
\end{center}

\begin{center}
\begin{tikzpicture}[use Hobby shortcut]

\coordinate[label=right:\,\,\,\,$P_k$] (k1) at (0,0);
\coordinate (k2) at (0,2);
\coordinate[label=right:\,\,$P_1$] (k3) at (0,4);
\node[point] at (k1) {};
\node[point] at (k3) {};

\coordinate[label=left:$C'_0$] (C01) at (-4,6);
\coordinate (C03) at (1,2.5);
\coordinate[label=right:$C'_0$] (C05) at (1,2);
\coordinate (C04) at (1,1.5);
\coordinate[label=left:$C'_0$] (C06) at (-6,-1.5);
\coordinate (C07) at (-6,-3);
\coordinate (C09) at (-6,-3.65);
\coordinate (C08) at (-6,-4.3);
\coordinate[label=left:$C'_0$] (C02) at (-6,-8.5);

\coordinate (C09A) at (-6,-3.85);
\coordinate (C09B) at (-6,-3.45);

\coordinate (Ca10A) at (-6,-6.85);
\coordinate (Ca10B) at (-6,-6.45);

\draw[thin,black,out angle=90,in angle=-90] (C02)..(Ca10A);
\draw[dotted,black,out angle=90,in angle=-90] (Ca10A)..(Ca10B);
\draw[thin,black,out angle=90,in angle=-90] (Ca10B)..(C08)..(C09A);
\draw[dotted,black,out angle=90,in angle=-90]  (C09A)..(C09B);
\draw[thin,black,out angle=90,in angle=-90] (C09B)..(C07)..(C06)..(k1)..(C04);
\draw [dotted, black, out angle=90,in angle=-90] (C04)..(C03);
\draw [thin, black, out angle=90,in angle=-20](C03)..(k3)..(C01);


\coordinate[label=left:$C'_a$] (Ca1) at (-2,6);
\coordinate (Ca3) at (2.5,2.5);
\coordinate[label=right:$C'_a$] (Ca14) at (2.6,2);
\coordinate (Ca4) at (2.5,1.5);
\coordinate (Ca5) at (-1,-1);
\coordinate[label=left:$C'_a$] (Ca13) at (-2,-1.5);
\coordinate (Ca11) at (-5.5,-2.7);
\coordinate (Ca10) at (-6,-3);
\coordinate[label=right:$C'_a$] (Ca6) at (-5.7,-3.65);
\coordinate (Ca7) at (-6,-4.3);
\coordinate (Ca12) at (-5.7,-5.4);
\coordinate (Ca13) at (-6,-6);
\coordinate[label=right:$C'_a$] (Ca14) at (-5.7,-6.65);
\coordinate (Ca15) at (-6,-7.3);
\coordinate (Ca16) at (-5.7,-7.7);
\coordinate[label=right:$C'_a$] (Ca2) at (-3,-8.5);

\coordinate (Ca6A) at (-5.7,-3.85);
\coordinate (Ca6B) at (-5.7,-3.45);

\coordinate (Ca14A) at (-5.7,-6.85);
\coordinate (Ca14B) at (-5.7,-6.45);

\coordinate [label=right:$0<a<\frac{n}{2}$] (Calab) at (-3.5,6.5);

\draw[thin,black,out angle=110,in angle=-90] (Ca2)..(Ca16)..(Ca15)..(Ca14A);
\draw[dotted,black,out angle=90,in angle=-90] (Ca14A)..(Ca14B);
\draw[thin,black,out angle=90,in angle=-90](Ca14B)..(Ca13)..(Ca12)..(Ca7);
\draw[thin,black,out angle=90,in angle=-90](Ca7)..(Ca6A);
\draw[dotted,black,out angle=90,in angle=-90] (Ca6A)..(Ca6B);
\draw[thin,black,out angle=90,in angle=-90](Ca6B)..(Ca10)..(Ca11)..(Ca5)..(k1)..(Ca4);
\draw[dotted,black,out angle=90,in angle=-90](Ca4)..(Ca3);
\draw[thin,black,out angle=90,in angle=-20](Ca3)..(k3)..(Ca1);


\coordinate[label=right:$C'_b$] (Cb1) at (2,6);
\coordinate (Cb3) at (-2.5,2.5);
\coordinate[label=right:$C'_b$] (Cb14) at (-2.6,2);
\coordinate (Cb4) at (-2.5,1.5);
\coordinate (Cb5) at (1,-1);
\coordinate[label=left:$C'_b$] (Cb13) at (3,-1.5);
\coordinate (Cb11) at (6.5,-2.7);
\coordinate (Cb10) at (7,-3);
\coordinate[label=right:$C'_b$] (Cb6) at (6,-3.65);
\coordinate (Cb7) at (7,-4.3);
\coordinate (Cb12) at (6.7,-5.4);
\coordinate (Cb13) at (7,-6);
\coordinate[label=right:$C'_b$] (Cb14) at (6,-6.65);
\coordinate (Cb15) at (7,-7.3);
\coordinate (Cb16) at (6.5,-7.7);
\coordinate[label=right:$C'_b$] (Cb2) at (4,-8.5);

\coordinate (Cb6A) at (6.7,-3.85);
\coordinate (Cb6B) at (6.7,-3.45);

\coordinate (Cb14A) at (6.7,-6.85);
\coordinate (Cb14B) at (6.7,-6.45);

\coordinate [label=left:$0<b<\frac{n}{2}$] (Calab) at (3.5,6.5);

\draw[thin,black,out angle=70,in angle=-90] (Cb2)..(Cb16)..(Cb15)..(Cb14A);
\draw[dotted,black,out angle=90,in angle=-90] (Cb14A)..(Cb14B);
\draw[thin,black,out angle=90,in angle=-90](Cb14B)..(Cb13)..(Cb12)..(Cb7);
\draw[thin,black,out angle=90,in angle=-90](Cb7)..(Cb6A);
\draw[dotted,black,out angle=90,in angle=-90] (Cb6A)..(Cb6B);
\draw[thin,black,out angle=90,in angle=-90](Cb6B)..(Cb10)..(Cb11)..(Cb5)..(k1)..(Cb4);
\draw[dotted,black,out angle=90,in angle=-90](Cb4)..(Cb3);
\draw[thin,black,out angle=90,in angle=200](Cb3)..(k3)..(Cb1);

\coordinate[label=right:$C'_n$] (Cn1) at (4,6);
\coordinate (Cn3) at (-1,2.5);
\coordinate[label=left:$C'_n$] (Cn5) at (-1,2);
\coordinate (Cn4) at (-1,1.5);
\coordinate[label=right:$C'_n$] (Cn6) at (7,-1.5);
\coordinate (Cn7) at (7,-3);
\coordinate (Cn9) at (7,-3.65);
\coordinate (Cn8) at (7,-4.3);
\coordinate[label=right:$C'_n$] (Cn2) at (7,-8.5);

\coordinate (Cn9A) at (7,-3.85);
\coordinate (Cn9B) at (7,-3.45);

\coordinate (Cn10A) at (7,-6.85);
\coordinate (Cn10B) at (7,-6.45);

\draw[thin,black,out angle=90,in angle=-90] (Cn2)..(Cn10A);
\draw[dotted,black,out angle=90,in angle=-90] (Cn10A)..(Cn10B);
\draw[thin,black,out angle=90,in angle=-90] (Cn10B)..(Cn8)..(Cn9A);
\draw[dotted,black,out angle=90,in angle=-90]  (Cn9A)..(Cn9B);
\draw[thin,black,out angle=90,in angle=-90] (Cn9B)..(Cn7)..(Cn6)..(k1)..(Cn4);
\draw [dotted, black, out angle=90,in angle=-90] (Cn4)..(Cn3);
\draw [thin, black, out angle=90,in angle=200](Cn3)..(k3)..(Cn1);



\coordinate (Ei0) at (-7,-3);
\coordinate (Ei1) at (-6,-3);
\node[point] at (Ei1) {};
\coordinate (Ei3) at (7,-3);
\node[point] at (Ei3) {};
\coordinate[label=right:$E'_{\infty,1}$] (Ei4) at (8,-3);
\draw[thin,black,out angle=0,in angle=180] (Ei0)..(Ei1)..(Ei3)..(Ei4);

\coordinate[label=right:$\vdots$] (Q3) at (8,-3.5);

\coordinate (Ef0) at (-7,-4.3);
\coordinate (Ef1) at (-6,-4.3);
\node[point] at (Ef1) {};
\coordinate (Ef3) at (7,-4.3);
\node[point] at (Ef3) {};
\coordinate[label=right:$E'_{\infty,{\varepsilon_2(M)}}$] (Ef4) at (8,-4.3);
\draw[thin,black,out angle=0,in angle=180] (Ef0)..(Ef1)..(Ef3)..(Ef4);



\coordinate (F01a) at (-7,-5.83);
\coordinate (F01b) at (2,-7.4);
\draw[thin,black, name path=F01] plot[smooth] coordinates{(F01a) (F01b)};

\coordinate (Fn1a) at (8,-5.83);
\coordinate (Fn1b) at (-1,-7.4);
\draw[thin,black,name path=Fn1] plot[smooth] coordinates{(Fn1a) (Fn1b)};

\path [name intersections={of=F01 and Fn1,by=i1}];

\node[point] at (i1) {};

\coordinate (F0ea) at (-7,-7.14);
\coordinate (F0eb) at (2,-8.7);
\draw[thin,black, name path=F0e] plot[smooth] coordinates{(F0ea) (F0eb)};

\coordinate (Fnea) at (8,-7.14);
\coordinate (Fneb) at (-1,-8.7);
\draw[thin,black, name path=Fne] plot[smooth] coordinates{(Fnea) (Fneb)};

\path [name intersections={of=F0e and Fne,by=i2}];

\node[point] at (i2) {};

\coordinate[label=right:$F'_{0,1}$] (Fi4) at (-8,-6);
\coordinate (Fi1) at (-6,-6);
\node[point] at (Fi1) {};
\coordinate (Fi3) at (7,-6);
\node[point] at (Fi3) {};
\coordinate[label=right:$F'_{n,1}$] (Fi4) at (8,-6);

\coordinate[label=right:$\vdots$] (Q3) at (8,-6.5);
\coordinate[label=right:$\vdots$] (Q3) at (-8,-6.5);

\coordinate[label=right:$F'_{0,\varepsilon_3(M)}$] (Fi4) at (-8,-7.4);
\coordinate (Ff1) at (-6,-7.3);
\node[point] at (Ff1) {};
\coordinate (Ff3) at (7,-7.3);
\node[point] at (Ff3) {};
\coordinate[label=right:$F'_{n,{\varepsilon_3(M)}}$] (Ff4) at (8,-7.3);

\end{tikzpicture}

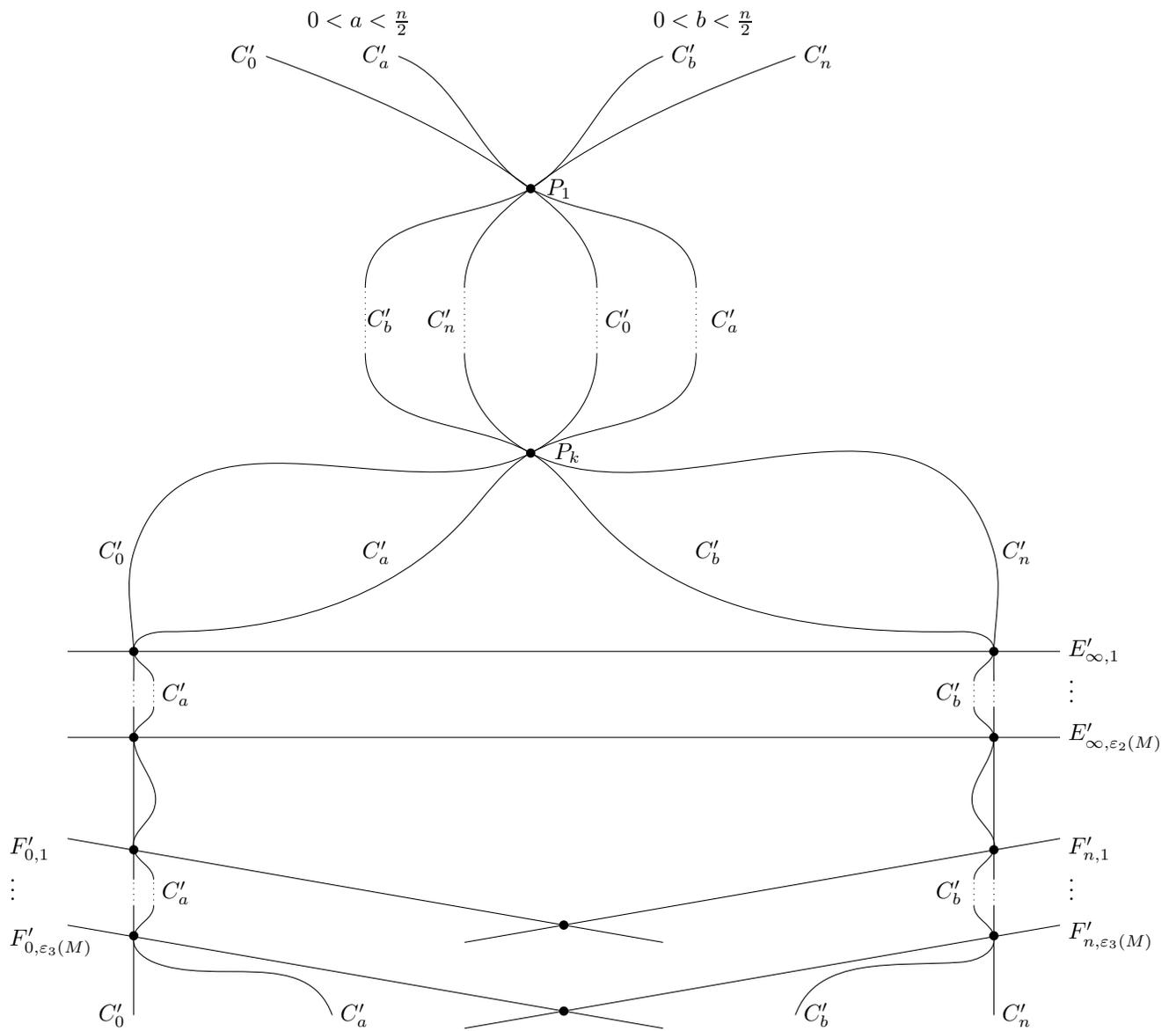
\captionof{figure}{Case $p\equiv 11 \mod 12$, $n$ odd.}
\label{fig:fig6}
\end{center}

\begin{center}
\begin{tikzpicture}[use Hobby shortcut]

\coordinate[label=right:\,\,\,\,$P_k$] (k1) at (0,0);
\coordinate (k2) at (0,2);
\coordinate[label=right:\,\,$P_1$] (k3) at (0,4);
\node[point] at (k1) {};
\node[point] at (k3) {};

\coordinate[label=left:$C'_0$] (C01) at (-4,6);
\coordinate (C03) at (1,2.5);
\coordinate[label=right:$C'_0$] (C05) at (1,2);
\coordinate (C04) at (1,1.5);
\coordinate[label=left:$C'_0$] (C06) at (-6,-1.5);
\coordinate (C07) at (-6,-3);
\coordinate (C09) at (-6,-3.65);
\coordinate (C08) at (-6,-4.3);
\coordinate[label=left:$C'_0$] (C02) at (-6,-8.5);

\coordinate (C09A) at (-6,-3.85);
\coordinate (C09B) at (-6,-3.45);

\coordinate (Ca10A) at (-6,-6.85);
\coordinate (Ca10B) at (-6,-6.45);

\draw[thin,black,out angle=90,in angle=-90] (C02)..(Ca10A);
\draw[dotted,black,out angle=90,in angle=-90] (Ca10A)..(Ca10B);
\draw[thin,black,out angle=90,in angle=-90] (Ca10B)..(C08)..(C09A);
\draw[dotted,black,out angle=90,in angle=-90]  (C09A)..(C09B);
\draw[thin,black,out angle=90,in angle=-90] (C09B)..(C07)..(C06)..(k1)..(C04);
\draw [dotted, black, out angle=90,in angle=-90] (C04)..(C03);
\draw [thin, black, out angle=90,in angle=-20](C03)..(k3)..(C01);


\coordinate[label=left:$C'_a$] (Ca1) at (-2,6);
\coordinate (Ca3) at (2.5,2.5);
\coordinate[label=right:$C'_a$] (Ca14) at (2.6,2);
\coordinate (Ca4) at (2.5,1.5);
\coordinate (Ca5) at (-1,-1);
\coordinate[label=left:$C'_a$] (Ca13) at (-2,-1.5);
\coordinate (Ca11) at (-5.5,-2.7);
\coordinate (Ca10) at (-6,-3);
\coordinate[label=right:$C'_a$] (Ca6) at (-5.7,-3.65);
\coordinate (Ca7) at (-6,-4.3);
\coordinate (Ca12) at (-5.7,-5.4);
\coordinate (Ca13) at (-6,-6);
\coordinate[label=right:$C'_a$] (Ca14) at (-5.7,-6.65);
\coordinate (Ca15) at (-6,-7.3);
\coordinate (Ca16) at (-5.7,-7.5);
\coordinate[label=right:$C'_a$] (Ca2) at (-3,-8.5);

\coordinate (Ca6A) at (-5.7,-3.85);
\coordinate (Ca6B) at (-5.7,-3.45);

\coordinate (Ca14A) at (-5.7,-6.85);
\coordinate (Ca14B) at (-5.7,-6.45);

\coordinate [label=right:$0<a<\frac{n}{2}$] (Calab) at (-3.5,6.5);

\draw[thin,black,out angle=90,in angle=-90] (Ca2)..(Ca16)..(Ca15)..(Ca14A);

\draw[dotted,black,out angle=90,in angle=-90] (Ca14A)..(Ca14B);

\draw[thin,black,out angle=90,in angle=-90](Ca14B)..(Ca13)..(Ca12)..(Ca7);
\draw[thin,black,out angle=90,in angle=-90](Ca7)..(Ca6A);
\draw[dotted,black,out angle=90,in angle=-90] (Ca6A)..(Ca6B);
\draw[thin,black,out angle=90,in angle=-90](Ca6B)..(Ca10)..(Ca11)..(Ca5)..(k1)..(Ca4);
\draw[dotted,black,out angle=90,in angle=-90](Ca4)..(Ca3);
\draw[thin,black,out angle=90,in angle=-20](Ca3)..(k3)..(Ca1);


\coordinate[label=right:$C'_{\frac{n}{2}}$] (Ca21) at (0,6);
\coordinate (Ca23) at (0,2.5);
\coordinate[label=right:$C'_{\frac{n}{2}}$] (Ca24) at (-0.05,2);
\coordinate[label=right:$C'_{\frac{n}{2}}$] (Ca25) at (0,-3.65);
\coordinate[label=right:$C'_{\frac{n}{2}}$] (Ca22) at (0,-8.5);

\coordinate (CaExtra1) at (0, 1.5);
\coordinate (CaExtra2) at (0, 2.5);

\coordinate (Ca25A) at (0,-3.85);
\coordinate (Ca25B) at (0,-3.45);

\coordinate (Caa2) at (0,-6.9);
\coordinate (Caa3) at (0,-6.5);

\draw[thin,black] plot[smooth] coordinates{(Ca22) (Caa2)};
\draw[dotted, black] plot[smooth] coordinates{(Caa2) (Caa3)};
\draw[thin, black] plot[smooth] coordinates{(Caa3) (Ca25A)};
\draw[dotted,black] plot[smooth] coordinates{(Ca25A) (Ca25B)};
\draw[thin, black] plot[smooth] coordinates{(Ca25B) (CaExtra1)};
\draw[dotted, black] plot[smooth] coordinates{(CaExtra1) (CaExtra2)};
\draw[thick, black] plot[smooth] coordinates{(CaExtra2) (Ca21)};


\coordinate[label=right:$C'_b$] (Cb1) at (2,6);
\coordinate (Cb3) at (-2.5,2.5);
\coordinate[label=right:$C'_b$] (Cb14) at (-2.6,2);
\coordinate (Cb4) at (-2.5,1.5);
\coordinate (Cb5) at (1,-1);
\coordinate[label=left:$C'_b$] (Cb13) at (3,-1.5);
\coordinate (Cb11) at (6.5,-2.7);
\coordinate (Cb10) at (7,-3);
\coordinate[label=right:$C'_b$] (Cb6) at (6,-3.65);
\coordinate (Cb7) at (7,-4.3);
\coordinate (Cb12) at (6.7,-5.4);
\coordinate (Cb13) at (7,-6);
\coordinate[label=right:$C'_b$] (Cb14) at (6,-6.65);
\coordinate (Cb15) at (7,-7.3);
\coordinate (Cb16) at (6.7,-7.5);
\coordinate[label=right:$C'_b$] (Cb2) at (4,-8.5);

\coordinate (Cb6A) at (6.7,-3.85);
\coordinate (Cb6B) at (6.7,-3.45);

\coordinate (Cb14A) at (6.7,-6.85);
\coordinate (Cb14B) at (6.7,-6.45);

\coordinate [label=left:$0<b<\frac{n}{2}$] (Calab) at (3.5,6.5);

\draw[thin,black,out angle=90,in angle=-90] (Cb2)..(Cb16)..(Cb15)..(Cb14A);

\draw[dotted,black,out angle=90,in angle=-90] (Cb14A)..(Cb14B);

\draw[thin,black,out angle=90,in angle=-90](Cb14B)..(Cb13)..(Cb12)..(Cb7);
\draw[thin,black,out angle=90,in angle=-90](Cb7)..(Cb6A);
\draw[dotted,black,out angle=90,in angle=-90] (Cb6A)..(Cb6B);
\draw[thin,black,out angle=90,in angle=-90](Cb6B)..(Cb10)..(Cb11)..(Cb5)..(k1)..(Cb4);
\draw[dotted,black,out angle=90,in angle=-90](Cb4)..(Cb3);
\draw[thin,black,out angle=90,in angle=200](Cb3)..(k3)..(Cb1);

\coordinate[label=right:$C'_n$] (Cn1) at (4,6);
\coordinate (Cn3) at (-1,2.5);
\coordinate[label=left:$C'_n$] (Cn5) at (-1,2);
\coordinate (Cn4) at (-1,1.5);
\coordinate[label=right:$C'_n$] (Cn6) at (7,-1.5);
\coordinate (Cn7) at (7,-3);
\coordinate (Cn9) at (7,-3.65);
\coordinate (Cn8) at (7,-4.3);
\coordinate[label=right:$C'_n$] (Cn2) at (7,-8.5);

\coordinate (Cn9A) at (7,-3.85);
\coordinate (Cn9B) at (7,-3.45);

\coordinate (Cn10A) at (7,-6.85);
\coordinate (Cn10B) at (7,-6.45);

\draw[thin,black,out angle=90,in angle=-90] (Cn2)..(Cn10A);
\draw[dotted,black,out angle=90,in angle=-90] (Cn10A)..(Cn10B);
\draw[thin,black,out angle=90,in angle=-90] (Cn10B)..(Cn8)..(Cn9A);
\draw[dotted,black,out angle=90,in angle=-90]  (Cn9A)..(Cn9B);
\draw[thin,black,out angle=90,in angle=-90] (Cn9B)..(Cn7)..(Cn6)..(k1)..(Cn4);
\draw [dotted, black, out angle=90,in angle=-90] (Cn4)..(Cn3);
\draw [thin, black, out angle=90,in angle=200](Cn3)..(k3)..(Cn1);



\coordinate (Ei0) at (-7,-3);
\coordinate (Ei1) at (-6,-3);
\node[point] at (Ei1) {};
\coordinate (Ei2) at (0,-3);
\node[point] at (Ei2) {};
\coordinate (Ei3) at (7,-3);
\node[point] at (Ei3) {};
\coordinate[label=right:$E'_{\infty,1}$] (Ei4) at (8,-3);
\draw[thin,black,out angle=0,in angle=180] (Ei0)..(Ei1)..(Ei2)..(Ei3)..(Ei4);

\coordinate[label=right:$\vdots$] (Q3) at (8,-3.5);

\coordinate (Ef0) at (-7,-4.3);
\coordinate (Ef1) at (-6,-4.3);
\node[point] at (Ef1) {};
\coordinate (Ef2) at (0,-4.3);
\node[point] at (Ef2) {};
\coordinate (Ef3) at (7,-4.3);
\node[point] at (Ef3) {};
\coordinate[label=right:$E'_{\infty,{\varepsilon_2(M)}}$] (Ef4) at (8,-4.3);
\draw[thin,black,out angle=0,in angle=180] (Ef0)..(Ef1)..(Ef2)..(Ef3)..(Ef4);


\coordinate (Fi0) at (-7,-6);
\coordinate (Fi1) at (-6,-6);
\node[point] at (Fi1) {};
\coordinate (Fi2) at (0,-6);
\node[point] at (Fi2) {};
\coordinate (Fi3) at (7,-6);
\node[point] at (Fi3) {};
\coordinate[label=right:$F'_{\infty,1}$] (Fi4) at (8,-6);
\draw[thin,black,out angle=0,in angle=180] (Fi0)..(Fi1)..(Fi2)..(Fi3)..(Fi4);

\coordinate[label=right:$\vdots$] (Q3) at (8,-6.5);

\coordinate (Ff0) at (-7,-7.3);
\coordinate (Ff1) at (-6,-7.3);
\node[point] at (Ff1) {};
\coordinate (Ff2) at (0,-7.3);
\node[point] at (Ff2) {};
\coordinate (Ff3) at (7,-7.3);
\node[point] at (Ff3) {};
\coordinate[label=right:$F'_{\infty,{\varepsilon_3(M)}}$] (Ff4) at (8,-7.3);
\draw[thin,black,out angle=0,in angle=180] (Ff0)..(Ff1)..(Ff2)..(Ff3)..(Ff4);

\end{tikzpicture}

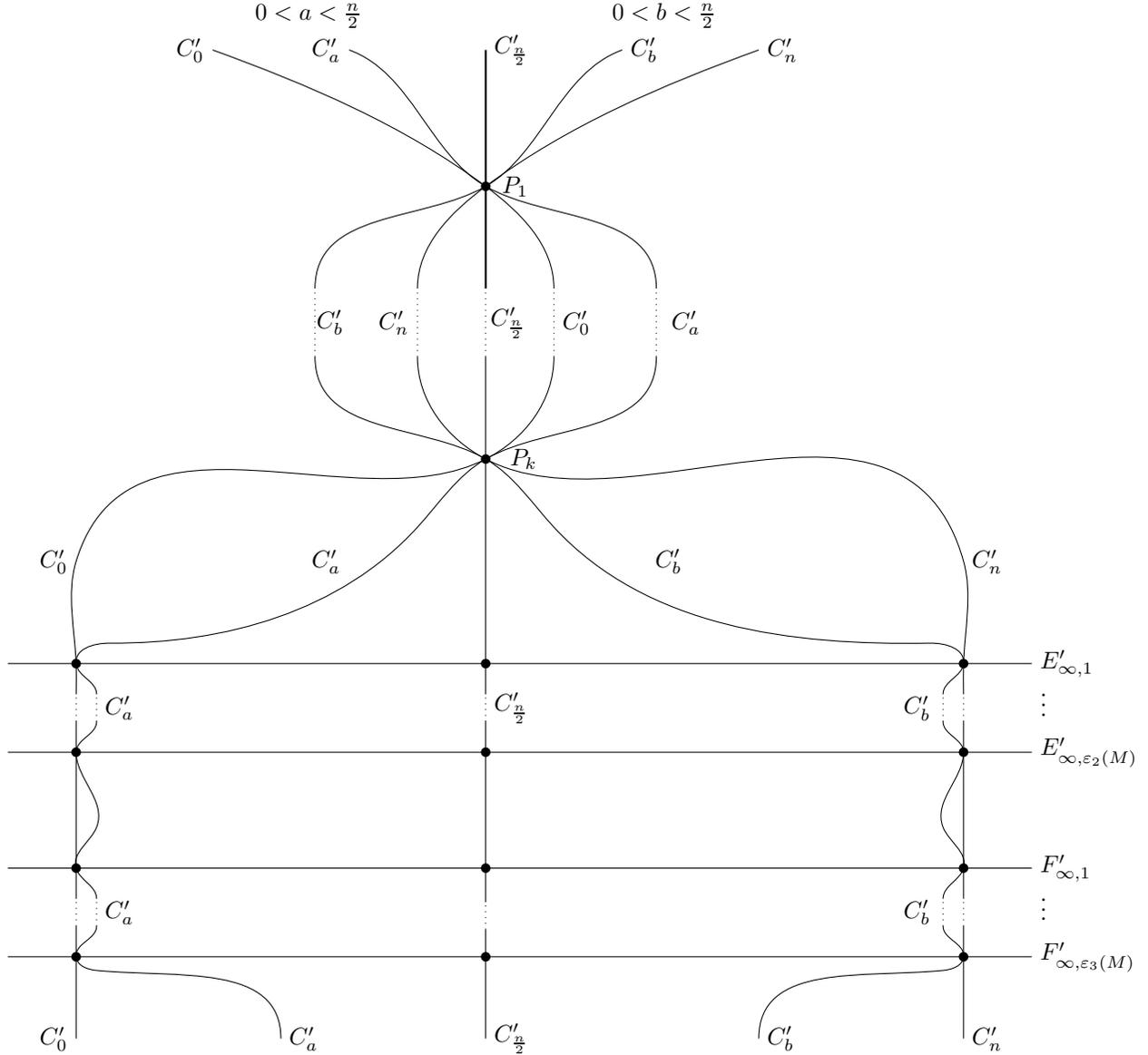
\captionof{figure}{Case $p\equiv 11 \mod 12$, $n$ even.}
\label{fig:fig5}
\end{center}

\end{appendices}

\bibliographystyle{amsalpha}
\bibliography{biblio.bib}

\newcommand{\etalchar}[1]{$^{#1}$}
\providecommand{\bysame}{\leavevmode\hbox to3em{\hrulefill}\thinspace}
\providecommand{\MR}{\relax\ifhmode\unskip\space\fi MR }
\providecommand{\MRhref}[2]{%
  \href{http://www.ams.org/mathscinet-getitem?mr=#1}{#2}
}
\providecommand{\href}[2]{#2}
\begin{thebibliography}{ECdJ{\etalchar{+}}11}

\bibitem[Ara74]{Ara74}
S.~J. Arakelov, \emph{Intersection theory of divisors on an arithmetic
  surface}, Mathematics of the USSR-Izvestiya \textbf{8} (1974), no.~6, 1167.

\bibitem[AU97]{AU97}
A.~Abbes and E.~Ullmo, \emph{Auto-intersection du dualisant relatif des courbes
  modulaires ${X}_0({N})$}, J. Reine Angew. Math. \textbf{484} (1997), 1--70.

\bibitem[BBC20]{BBC20}
D.~Banerjee, D.~Borah, and C.~Chaudhuri, \emph{Arakelov self-intersection
  numbers of minimal regular models of modular curves ${X}_0(p^2)$},
  Mathematische Zeitschrift \textbf{296} (2020).

\bibitem[BMC22]{BMC22}
D.~Banerjee, P.~Majumder, and C.~Chaudhuri, \emph{The intersection matrices of
  ${X}_0(p^r)$ and some applications}, 2022, arXiv:2210.08866.

\bibitem[CK09]{CK09}
C.~Curilla and U.~K{\"u}hn, \emph{On the arithmetic self-intersection numbers
  of the dualizing sheaf for fermat curves of prime exponent}, 2009.

\bibitem[dJ04]{dJ04}
R.~de~Jong, \emph{{E}xplicit {A}rakelov {G}eometry}, Ph.D. thesis, University
  of Amsterdam, 2004.

\bibitem[DM]{codes}
P.~Dolce and P.~Mercuri, \emph{{W}olfram {M}athematica shared code},
  \url{https://github.com/mercuri-pietro/Arakelov-theory-on-modular-curves},
  Accessed: 2023-06-19.

\bibitem[DR73]{DR73}
P.~Deligne and M.~Rapoport, \emph{Les sch{\'e}mas de modules de courbes
  elliptiques}, Modular Functions of One Variable II (Berlin, Heidelberg)
  (Pierre Deligne and Willem Kuijk, eds.), Springer Berlin Heidelberg, 1973,
  pp.~143--316.

\bibitem[Dri73]{Dri73}
V.G. Drinfeld, \emph{Two theorems on modular curves}, Funct. Anal. Appl.
  \textbf{7} (1973), no.~2, 155--156.

\bibitem[DS05]{DS05}
F.~Diamond and J.~Shurman, \emph{A first course in modular forms}, Graduate
  Texts in Mathematics, vol. 228, Springer-Verlag, New York, 2005.

\bibitem[ECdJ{\etalchar{+}}11]{ECdJMB11}
B.~Edixhoven, J.-M. Couveignes, R.~de~Jong, F.~Merkl, and J~Bosman (eds.),
  \emph{Computational aspects of modular forms and {G}alois representations},
  Annals of Mathematics Studies, vol. 176, Princeton University Press,
  Princeton, NJ, 2011.

\bibitem[Edi90]{Edi90}
B.~Edixhoven, \emph{{Minimal resolution and stable reduction of $X_0(N)$}},
  Annales de l'Institut Fourier \textbf{40} (1990), 31--67.

\bibitem[Fal84]{Fal84}
G.~Faltings, \emph{Calculus on arithmetic surfaces}, Ann. of Math. (2)
  \textbf{119} (1984), no.~2, 387--424.

\bibitem[GvP22]{GvP22}
M.~Grados and A.-M. von Pippich, \emph{{Self-intersection of the relative
  dualizing sheaf on modular curves $X(N)$}}, 2022.

\bibitem[Jav14]{Jav14}
A.~Javanpeykar, \emph{Polynomial bounds for {A}rakelov invariants of {B}elyi
  curves}, Algebra Number Theory \textbf{8} (2014), no.~1, 89--140, With an
  appendix by Peter Bruin.

\bibitem[KM85]{KM85}
N.~M. Katz and B.~Mazur, \emph{Arithmetic moduli of elliptic curves}, Annals of
  Mathematics Studies, vol. 108, Princeton University Press, Princeton, NJ,
  1985.

\bibitem[K{\"u}h13]{K13}
U.~K{\"u}hn, \emph{On the arithmetic self-intersection number of the dualizing
  sheaf on arithmetic surfaces}, 2013.

\bibitem[Lan88]{Lan88}
S.~Lang, \emph{Introduction to {A}rakelov theory}, Springer-Verlag, New York,
  1988.

\bibitem[Liu06]{Liu06}
Q.~Liu, \emph{Algebraic geometry and arithmetic curves}, paperback ed., Oxford
  University Press, USA, 8 2006 (English).

\bibitem[Man72]{Man72}
J.~Manin, \emph{Parabolic points and zeta-functions of modular curves},
  Mathematics of the USSR-Izvestiya \textbf{6} (1972), no.~1, 19--64.

\bibitem[May14]{May14}
H.~Mayer, \emph{Self-intersection of the relative dualizing sheaf on modular
  curves {$X_1(N)$}}, J. Th\'{e}or. Nombres Bordeaux \textbf{26} (2014), no.~1,
  111--161.

\bibitem[Maz77]{Maz77}
B.~Mazur, \emph{Modular curves and the eisenstein ideal}, Publications
  Mathématiques de l'IHÉS \textbf{47} (1977), 33--186 (eng).

\bibitem[MB89]{MB89}
L.~Moret-Bailly, \emph{La formule de {N}oether pour les surfaces
  arithm\'{e}tiques}, Invent. Math. \textbf{98} (1989), no.~3, 491--498.

\bibitem[MB90]{MB90}
\bysame, \emph{Hauteurs et classes de {C}hern sur les surfaces
  arithm\'{e}tiques}, no. 183, 1990, S\'{e}minaire sur les Pinceaux de Courbes
  Elliptiques (Paris, 1988), pp.~37--58.

\bibitem[Mor13]{Mor13}
A.~Moriwaki, \emph{Toward {D}irichlet’s unit theorem on arithmetic
  varieties}, Kyoto J. Math. \textbf{53} (2013), no.~1, 197--259.

\bibitem[Mor14]{Mor14}
\bysame, \emph{Arakelov geometry (translations of mathematical monographs)},
  hardcover ed., American Mathematical Society, 11 2014 (English).

\bibitem[MU98]{MU98}
P.~Michel and E.~Ullmo, \emph{Points de petite hauteur sur les courbes
  modulaires {$X_0(N)$}}, Invent. Math. \textbf{131} (1998), no.~3, 645--674.

\bibitem[MvP22]{MvP22}
P.~Majumder and A.-M. von Pippich, \emph{Bounds for canonical {G}reen's
  functions at cusps}, 2022, arXiv:2210.04452.

\bibitem[NR83]{NR83}
J.-L. Nicolas and G.~Robin, \emph{Majorations explicites pour le nombre de
  diviseurs de {$N$}}, Canad. Math. Bull. \textbf{26} (1983), no.~4, 485--492.

\bibitem[{Sta}18]{stacks}
The {Stacks Project Authors}, \emph{\textit{Stacks Project}},
  \url{https://stacks.math.columbia.edu}, 2018.

\bibitem[Ull98]{Ull98}
E.~Ullmo, \emph{Positivit\'{e} et discr\'{e}tion des points alg\'{e}briques des
  courbes}, Ann. of Math. (2) \textbf{147} (1998), no.~1, 167--179.

\bibitem[Zha93]{Zha93}
S.~Zhang, \emph{Admissible pairing on a curve}, Invent. Math. \textbf{112}
  (1993), no.~1, 171--193.

\end{thebibliography}

\Addresses

\end{document}